\documentclass[reqno]{amsart}
\usepackage{amssymb,amsmath,amsfonts,amscd,amsthm}
\usepackage{amsbsy,bm}
\makeatletter
\@namedef{subjclassname@2020}{\textup{2020} Mathematics Subject Classification}
\makeatother
\usepackage[a4paper,textwidth=13.5cm,centering]{geometry}

\newtheorem{theorem}{Theorem}[section]
\newtheorem{definition}{Definition}[section]
\newtheorem{lemma}{Lemma}[section]
\newtheorem{proposition}{Proposition}[section]
\newtheorem{corollary}{Corollary}[section]
\newtheorem{remark}{Remark}[section]

\newtheorem{example}{Example}[section]

\newtheorem*{chernconjecture}{Chern conjecture}
\newtheorem*{calabiconjecture}{Calabi conjecture}
\newtheorem*{conjecture}{Centroaffine Bernstein conjecture}

\numberwithin{equation}{section}

\begin{document}
\title[Complete centroaffine extremal hypersurfaces]
{Complete centroaffine extremal hypersurfaces and
centroaffine Bernstein conjecture}

\author{Cheng Xing}
\address{
School of Mathematics and Statistics,
Henan Normal University, Xinxiang 453007,
People's Republic of China}
\email{xingcheng@nankai.edu.cn}

\author{Yalin Sun}
\address{
School of Mathematics and Statistics,
Henan Normal University, Xinxiang 453007,
People's Republic of China}
\email{ylsun2024@126.com}

\author{Ruiwei Xu}
\thanks{R. Xu is the corresponding author.}
\address{
School of Mathematics and Statistics,
Henan Normal University, Xinxiang 453007,
People's Republic of China}
\email{rwxu@htu.edu.cn}



\date{}

\begin{abstract}
The centroaffine Bernstein conjecture for hyperbolic centroaffine extremal hypersurfaces predicts that
nonnegative Ricci curvature together with either Euclidean or centroaffine completeness forces the
hypersurface to lie in Wang's class. We disprove it by systematically studying Calabi compositions:
we develop a general geometric method for constructing complete centroaffine extremal hypersurfaces
and apply it to construct two families that violate this rigidity. The first family is Euclidean complete
but has incomplete centroaffine metric, nonnegative Ricci curvature, and positive Gauss curvature in
dimension two; iterating the composition with a point gives examples in every dimension $n\ge2$.
The second family is flat and complete in both metrics in dimension three, and iteration yields
examples in every dimension $n\ge3$. Thus the centroaffine Bernstein conjecture fails under Euclidean
completeness for $n\ge2$ and under centroaffine completeness for $n\ge3$.
\end{abstract}

\maketitle

\section{Introduction}\label{sect:1}

The classical Bernstein problem asks whether every entire solution of the
minimal surface equation on $\mathbb R^n$ must be affine linear, or equivalently,
whether every entire minimal graph in $\mathbb R^{n+1}$ must be a hyperplane.
Bernstein \cite{Ber15} proved the assertion for $n=2$, and De Giorgi~\cite{DG65},
Almgren~\cite{Alm66}, and Simons~\cite{Sim68} extended it to $3\leq n\leq7$.
Bombieri--De Giorgi--Giusti \cite{BDGG69} constructed a nonaffine entire minimal
graph for $n=8$, giving counterexamples in all higher dimensions as well. These
results show that the rigidity of entire minimal graphs depends on dimension.
Related Bernstein problems aim to classify complete or stable solutions of other
geometric equations. For instance, Chodosh--Li \cite{ChL24} resolved the stable
Bernstein problem in dimension three, while Mooney--Yang~\cite{MY24} completed
the anisotropic Bernstein problem by constructing nonlinear entire graphs over
$\mathbb R^4$. For further developments in related directions, see \cite{AC00,ChL24,DJX24,HLW26,JWZ23,JX02,
LLY20,MY24,Sav09} and the references therein.

Affine maximal surfaces are stationary points of the unimodular affine area
functional. Chern \cite{Che79} assumed the entire graph condition, which implies
Euclidean completeness, whereas Calabi \cite{Cal82,Cal88} assumed affine
completeness. Both conjectures assert that an affine maximal surface must be an
elliptic paraboloid.

\begin{chernconjecture}
Any locally strongly convex affine maximal surface $F:\mathbb R^2\to\mathbb R^3$
that is an entire graph must be an elliptic paraboloid.
\end{chernconjecture}

\begin{calabiconjecture}
Any locally strongly convex affine maximal surface $F:M^2\to\mathbb R^3$ with
complete Blaschke metric must be an elliptic paraboloid.
\end{calabiconjecture}

Calabi \cite{Cal82,Cal88} developed the variational framework and obtained
rigidity under both completeness assumptions. Trudinger--Wang \cite{TW00} proved
the Chern conjecture, while Li--Jia \cite{LJ01} solved the Calabi conjecture.
Li--Jia \cite{LJ09} obtained related rigidity results for the affine maximal type
equation, giving an analytic proof of the Chern conjecture. Thus the
two-dimensional affine Bernstein problem is settled. For $n\geq3$, the general
smooth Bernstein problem remains open.

The Bernstein problem for affine maximal surfaces in affine geometry has a
natural centroaffine analogue. In this setting, Wang \cite{Wan94} introduced the
Tchebychev operator $\mathcal T:=\widehat\nabla T$ and derived the variational
equation $\operatorname{Trace}(\mathcal T)=0$ for the centroaffine area
functional. Locally strongly convex solutions of this equation are called
\emph{centroaffine extremal hypersurfaces}; the relevant definitions are recalled
in Section \ref{sect:2}. For such hypersurfaces, Li--Li--Simon \cite{LLS04}
formulated five Bernstein problems and the following conjecture.

\begin{conjecture}[\cite{LLS04}]
Let $F:M^n\to\mathbb R^{n+1}$ $(n\ge2)$ be a hyperbolic centroaffine extremal
hypersurface. Assume that $F(M^n)$ satisfies at least one of the following two
completeness conditions: the Riemannian metric $g$ induced by the Euclidean metric
on $\mathbb R^{n+1}$ is complete, or the centroaffine metric $h$ is complete. If the
Ricci curvature of $h$ is nonnegative, then $F(M^n)$ is centroaffinely equivalent to
one of the following hypersurfaces:
\begin{equation}\label{eqn:1.1}
x_1^{\alpha_1}x_2^{\alpha_2}\cdots x_{n+1}^{\alpha_{n+1}}=1,
\ \  \alpha_1,\ldots,\alpha_{n+1}>0.
\end{equation}
\end{conjecture}

Partial affirmative results were obtained under additional assumptions.
Liu--Wang \cite{LW95} proved the conjecture for surfaces with complete flat
centroaffine metric. Li--Li--Simon \cite{LLS04} proved it, under centroaffine
completeness, when the length of the Tchebychev vector field is constant,
or, in dimension two, bounded. In a complementary direction, Lei--Xu--Zhao~\cite{LXZ25}
introduced the type II Calabi composition and constructed examples that answer
the third and fifth Bernstein problems. Using this composition and direct
hyperbolic constructions, Sun--Xing--Xu~\cite{SXX26} produced further complete
examples that solved all five problems. The example obtained from the type II
Calabi composition is elliptic, whereas their hyperbolic examples have negative
Ricci curvature. Hence these constructions do not settle the conjecture.

Motivated by the conjecture and the above examples, we systematically study
Calabi compositions to develop geometric methods for constructing complete
centroaffine extremal hypersurfaces. The resulting hypersurfaces, generated from
centroaffine curves via Calabi compositions, possess nonnegative Ricci curvature
and the required completeness. Proposition \ref{pro:4.1} implies that if two
centroaffine curves are used to construct a centroaffine hypersurface as in
Example 1.2, then the resulting hypersurface must have dimension $n=3$.
Proposition \ref{pro:5.1} implies that if a hyperbolic extremal surface is
obtained by the Calabi composition of a curve with a point, then the curve is
locally centroaffinely equivalent to a member of Wang's class; see Proposition
\ref{pro:4.2}(1). In this case, the resulting surface lies precisely in Wang's
class (\ref{eqn:1.1}). These facts lead us to construct a Euclidean complete
centroaffine extremal surface with positive Gauss curvature, as
in Example 1.1.

\begin{example}\label{exa:1.1}
Let $Q$ be the centroaffine curve defined in \eqref{eqn:4.25}. For $u\in\mathbb
R$, set
\begin{equation}\label{eqn:1.2}
\begin{aligned}
\Phi(u)=\int_{-\infty}^{u}e^{-\zeta^2}\,d\zeta,\ \
w(u)=\tfrac{e^{-u^2}}{\Phi(u)},\ \
a(u)=2u+w(u),\ \
b(u)=\tfrac1{\Phi(u)}.
\end{aligned}
\end{equation}
For $v\in\mathbb R$, set $\gamma(v)=Q(v/\sqrt2)$. Then the smooth map
$F_2:\mathbb R^2\to\mathbb R^3$, defined by
\begin{equation}\label{eqn:1.3}
\begin{aligned}
F_2(u,v)
&=\bigl(a(u)\gamma(v),\,b(u)\bigr)\\
&=\bigl(a(u)(e^{-v^2/2}+\sqrt2\,v f(\tfrac{v}{\sqrt2})),\,
\tfrac{a(u)v}{\sqrt2},\,b(u)\bigr),
\end{aligned}
\end{equation}
is a locally strongly convex centroaffine surface.
\end{example}

\begin{theorem}\label{thm:1.1}
The map $F_2:\mathbb R^2\to\mathbb R^3$ constructed in Example \ref{exa:1.1}
is a smooth proper embedding and a hyperbolic centroaffine extremal surface. Its induced Euclidean metric $g$ is complete, whereas its
centroaffine metric $h$ is incomplete and has positive Gauss curvature.
\end{theorem}

\begin{remark}\label{rem:1.1}
Geometrically, $F_2$ is obtained by rescaling the centroaffine curve $\gamma$
along a transverse direction. The functions $a(u),b(u)$ are chosen so that the centroaffine metric $h$
is conformal to the Euclidean metric $du^2+dv^2/2$ and $\log\psi=-\log4-2u^2+v^2$, so that the quadratic
terms cancel under $\Delta_h$.
\end{remark}

In dimension three, the following construction yields a counterexample with a
flat centroaffine metric and with both metrics complete. It also serves as the
starting point for a higher-dimensional family.

\begin{example}\label{exa:1.2}
Let $P,Q:\mathbb R\to\mathbb R^2$ be the centroaffine curves defined in
\eqref{eqn:4.24} and \eqref{eqn:4.25}, respectively. Then the smooth map
$F_3:\mathbb R^3\to\mathbb R^4$, defined by
\begin{equation}\label{eqn:1.4}
\begin{aligned}
F_3(t,u,v)
&=\bigl(e^tP(u),\,e^{-t}Q(v)\bigr)\\
&=\bigl(e^{t+u^2},\,e^{t+u^2}f(u),\,
e^{-t}(e^{-v^2}+2vf(v)),\,e^{-t}v\bigr),
\end{aligned}
\end{equation}
is a locally strongly convex centroaffine hypersurface.
\end{example}

\begin{theorem}\label{thm:1.2}
The map $F_3:\mathbb R^3\to\mathbb R^4$ constructed in Example \ref{exa:1.2}
is a smooth proper embedding and a hyperbolic centroaffine extremal hypersurface. Its centroaffine metric $h$ is flat and complete,
and its induced Euclidean metric $g$ is also complete.
\end{theorem}

\begin{remark}\label{rem:1.2}
The complete curves $P$ and $Q$ respectively satisfy $\Delta_{h_P}\log\psi_P
=-2$ and $\Delta_{h_Q}\log\psi_Q=2$. Moreover, their Calabi composition with
$\lambda=1$ cancels these contributions and gives $h=dt^2+du^2+dv^2$. Hence
the resulting hypersurface is extremal with flat complete centroaffine metric,
and properness gives Euclidean completeness.
\end{remark}

Now we extend the above two constructions to higher dimensions using the
Calabi composition with a point. Theorems \ref{thm:1.3} and \ref{thm:1.4} follow
by repeatedly Calabi composing $F_2$ and $F_3$, respectively, with a point and
taking $\lambda=1$. Proposition \ref{pro:5.1} preserves extremality, proper
embedding, Euclidean completeness, and the completeness type of the centroaffine
metric. It also preserves the nonnegative Ricci curvature of $F_2$ and the
flatness of $F_3$.

Starting with the surface in Example \ref{exa:1.1}, Proposition \ref{pro:5.1}
yields Euclidean complete hyperbolic centroaffine extremal hypersurfaces whose
centroaffine metrics have nonnegative Ricci curvature.

\begin{theorem}\label{thm:1.3}
Let $G_2=F_2:\mathbb R^2\to\mathbb R^3$ be the smooth map defined in Example
\ref{exa:1.1}. For $n\ge3$, write $\mathbb R^n=\mathbb R\times\mathbb R^{n-1}$
with coordinates $(r_n,p)$. Then the smooth map $G_n:\mathbb R^n\to\mathbb R^{n+1}$,
defined recursively by
\begin{equation}\label{eqn:1.5}
G_n(r_n,p)=(e^{r_n}G_{n-1}(p),\,e^{-r_n}),
\end{equation}
is a locally strongly convex smooth proper embedding and a hyperbolic centroaffine extremal hypersurface.
Its centroaffine metric $h$ has nonnegative Ricci curvature but is incomplete, whereas its induced
Euclidean metric $g$ is complete.
\end{theorem}

Starting instead with the hypersurface in Example~\ref{exa:1.2}, the same
composition preserves flatness and both completeness properties, providing
the following family.

\begin{theorem}\label{thm:1.4}
Let $F_3:\mathbb R^3\to\mathbb R^4$ be the smooth map defined in Example
\ref{exa:1.2}. For $n\ge4$, write $\mathbb R^n=\mathbb R\times\mathbb R^{n-1}$
with coordinates $(s_n,p)$. Then the smooth map $F_n:\mathbb R^n\to\mathbb
R^{n+1}$, defined recursively by
\begin{equation}\label{eqn:1.6}
F_n(s_n,p)=(e^{s_n}F_{n-1}(p),\,e^{-s_n}),
\end{equation}
is a locally strongly convex smooth proper embedding and a hyperbolic centroaffine extremal hypersurface. Its centroaffine metric
$h$ is flat and complete, and its induced Euclidean metric $g$ is also complete.
\end{theorem}

As noted in Section \ref{sect:5.4}, Wang's class has a flat centroaffine metric
and $\mathcal T=0$ \cite[Example 3.3]{LLS04}. In contrast, the family
$\{G_n\}_{n\ge2}$ is nonflat, while $\{F_n\}_{n\ge3}$ satisfies
$\mathcal T\ne0$ by
\eqref{eqn:5.28}. This contrast yields the following disproof of the centroaffine
Bernstein conjecture.

\begin{corollary}\label{cor:1.1}
The centroaffine Bernstein conjecture fails under Euclidean completeness for every $n\ge2$ and under
centroaffine completeness for every $n\ge3$. More precisely, in the first family, the metric $h$ is
incomplete with nonnegative Ricci curvature, and its Gauss curvature is positive when $n=2$;
in the second family, $h$ is flat and both metrics are complete.
\end{corollary}

Finally, we note that the case of a centroaffine complete hyperbolic extremal surface with
nonnegative Gauss curvature remains open. In a forthcoming paper, we will investigate
whether such a surface must belong to Wang's class; we expect this to be true.

The paper is organized as follows. Section \ref{sect:2} recalls the centroaffine
structure equations, Calabi compositions, and the relevant completeness notions.
In Section \ref{sect:3}, we explain the geometric construction of $F_2$ and
establish its curvature, extremality, embedding, and completeness properties,
thereby proving Theorem \ref{thm:1.1}. Section \ref{sect:4} develops the
two-factor composition, classifies the centroaffine curves needed to construct
$F_3$, and proves Theorem \ref{thm:1.2}. Finally, Section \ref{sect:5} studies
composition with a point and establishes its inheritance properties. Iteration
yields Theorems \ref{thm:1.3} and \ref{thm:1.4}, and the comparison with Wang's
class proves Corollary \ref{cor:1.1}.

\section{Preliminaries}\label{sect:2}

In this section, we give a brief review of the local theory of centroaffine
hypersurfaces in affine differential geometry, and recall some results on
such hypersurfaces and Riemannian manifolds. We assume throughout that
all hypersurface immersions have connected domains. For more details, we
refer to \cite{LLS04,LSW96} and the monographs \cite{LSZH15,NS94}.

\subsection{Local theory of centroaffine hypersurfaces}\label{sect:2.1}~

\vskip1mm
Let $\mathbb{R}^{n+1}$ denote the $(n+1)$-dimensional affine space
endowed with the standard flat connection $D$. A non-degenerate
hypersurface immersion $F:M^n\to\mathbb{R}^{n+1}$ is said to be
a {\it centroaffine hypersurface} if, at every $x\in M^n$, the
position vector $F(x)$ is transversal to $F_*(T_xM^n)$. Thus $F$
determines the {\it centroaffine normal} up to orientation. Let
$U,X,Y,Z$ denote tangent vector fields on $M^n$. With the centroaffine
normal $-\varepsilon F$, we can write the Gauss formula as follows:
\begin{equation}\label{eqn:2.1}
D_XF_{*}(Y)=F_{*}(\nabla_XY)+h(X,Y)(-\varepsilon F),
\end{equation}
where $\varepsilon\in\{\pm1\}$. This induces the {\it centroaffine
connection} $\nabla$ and the {\it centroaffine metric} $h$ on $M^n$.
When $h$ is positive definite, the locally strongly convex hypersurface
is called {\it elliptic} or {\it hyperbolic} according to $\varepsilon=1$
or $\varepsilon=-1$, respectively. We will always drop $F_*$ when no
ambiguity can arise.

Denote by $\widehat{\nabla}$ the Levi-Civita connection of the centroaffine
metric $h$. Since both $\widehat{\nabla}$ and $\nabla$ have zero torsion,
the difference tensor $K$, defined by
\begin{equation}\label{eqn:2.2}
K(X,Y)=K_XY:=\nabla_XY-\widehat{\nabla}_XY,
\end{equation}
is symmetric in $X$ and $Y$. It is related to the totally symmetric
{\it cubic form} $C:=\nabla h$ by
\begin{equation}\label{eqn:2.3}
C(X,Y,Z)=-2h(K(X,Y),Z),
\end{equation}
which implies that the operator $K_X$ is symmetric with respect to $h$.
Furthermore, the {\it Tchebychev form} $T^\flat$ and its metric dual,
the {\it Tchebychev vector field} $T$, are defined by
\begin{equation}\label{eqn:2.4}
T^\flat(X)=\tfrac{1}{n}{\rm trace}_h\,K_X,\ \
h(T,X)=T^\flat(X).
\end{equation}
It is worth noting that if $T$ vanishes, or equivalently ${\rm trace}_h
\,K_X=0$ for any $X\in TM^n$, then $M^n$ reduces to a proper affine
hypersphere centered at the origin of $\mathbb{R}^{n+1}$ (cf. p.279 of
\cite{LSZH15}). As the Pick--Berwald theorem says, $K$ or $C$ vanishes
if and only if the centroaffine hypersurface is a hyperquadric centered
at the origin.

Next, in terms of the curvature tensor $\widehat{R}$ of $h$, the Gauss and
Codazzi equations are given by
\begin{align}
&\widehat{R}(X,Y)Z
=\varepsilon(h(Y,Z)X-h(X,Z)Y)-[K_X,K_Y]Z,\label{eqn:2.5}\\
&(\widehat{\nabla}K)(X,Y,Z)
=(\widehat{\nabla}K)(Y,X,Z),\label{eqn:2.6}
\end{align}
where
\begin{align}
&[K_X,K_Y]Z:=K_XK_YZ-K_YK_XZ,\label{eqn:2.7}\\
&\widehat{R}(X,Y)Z:=\widehat{\nabla}_X\widehat{\nabla}_YZ
-\widehat{\nabla}_Y\widehat{\nabla}_XZ-\widehat{\nabla}_{[X,Y]}Z,\label{eqn:2.8}\\
&(\widehat{\nabla}K)(X,Y,Z):=\widehat{\nabla}_XK(Y,Z)
-K(\widehat{\nabla}_XY,Z)-K(Y,\widehat{\nabla}_XZ).\label{eqn:2.9}
\end{align}
For $n\ge2$, contracting the Gauss equation \eqref{eqn:2.5} twice, we
deduce that
\begin{equation}\label{eqn:2.10}
n(n-1)\chi=R=n(n-1)(J+\varepsilon)-n^2\|T\|^2_h,
\end{equation}
where $\|T\|^2_h:=h(T,T)$, $J:=\tfrac{1}{n(n-1)}h(K,K)$ is the {\it
centroaffine Pick invariant}, and $\chi$ and $R$ denote the {\it normalized
scalar curvature} and the {\it scalar curvature} of $h$, respectively.

Moreover, the covariant derivative $\widehat\nabla T$ is central to our
hypotheses. Contracting the Codazzi equation \eqref{eqn:2.6} gives
\begin{equation}\label{eqn:2.11}
h(\widehat{\nabla}_XT,Y)=h(X,\widehat{\nabla}_YT),
\end{equation}
which means precisely that the Tchebychev form $T^\flat$ is closed,
and the endomorphism $\mathcal{T}:TM^n\rightarrow TM^n$, defined by
\begin{equation}\label{eqn:2.12}
\mathcal{T}(X):=\widehat{\nabla}_XT,\ \ X\in TM^n,
\end{equation}
is called the \emph{Tchebychev operator} of $F:M^n\to\mathbb{R}^{n+1}$
(cf. \cite{LW95}). Equation \eqref{eqn:2.11} shows that $\mathcal{T}$
is self-adjoint with respect to $h$. Following Wang~\cite{Wan94}, we
call $\mathcal{T}$ the \emph{centroaffine shape operator}, while $H=(1/n)
\operatorname{tr}_h\mathcal T$ is the \emph{centroaffine mean curvature}.
A centroaffine hypersurface is called \emph{Tchebychev} if $\mathcal{T}
=H\cdot\operatorname{id}$.

To express the Euler--Lagrange equation, we choose local coordinates $(w^1,
\ldots,w^n)$ and write $F_i=\partial F/\partial w^i$. The \emph{Tchebychev
function} is the globally defined positive function
\begin{equation}\label{eqn:2.13}
\psi=\tfrac{\det(h_{ij})}{[F_1,\ldots,F_n,F]^2},
\end{equation}
where $[v_1,\ldots,v_{n+1}]$ denotes the determinant of the indicated
ambient vectors. The ratio in \eqref{eqn:2.13} is independent of the
chosen coordinates. Taking the trace in \eqref{eqn:2.1}  gives
\begin{equation}\label{eqn:2.14}
T=-\tfrac{1}{2n}\operatorname{grad}_h\log\psi,
\end{equation}
and $T^\flat=-(2n)^{-1}d\log\psi$. Set $\Delta_h=\operatorname{tr}_h
\operatorname{Hess}_h$. Taking the divergence in \eqref{eqn:2.14} yields
\begin{equation}\label{eqn:2.15}
nH=\operatorname{tr}_h\mathcal{T}=-\tfrac{1}{2n}\Delta_h\log\psi.
\end{equation}
Accordingly, a centroaffine hypersurface is said to be \emph{centroaffine
extremal} if $H=0$, or equivalently if $\Delta_h\log\psi=0$. This agrees
with Wang's notion of centroaffine minimality.

More explicitly, for a relatively compact domain $\Omega\subset M^n$, the
centroaffine area functional is $\mathcal A_\Omega(F)=\int_\Omega dV_h$.
In the hyperbolic case, the radial variation $F_\tau=(1+\tau\varphi)F$,
with $\varphi\in C_c^\infty(\Omega)$, satisfies $\left.\partial_\tau
h_\tau\right|_{\tau=0}=\operatorname{Hess}_\nabla\varphi$. Integration
by parts shows that
\begin{equation}\label{eqn:2.16}
\left.\tfrac{d}{d\tau}\right|_{\tau=0}\mathcal A_\Omega(F_\tau)
=\tfrac n2\int_\Omega\varphi\,\operatorname{div}_hT\,dV_h
=-\tfrac14\int_\Omega\varphi\,\Delta_h\log\psi\,dV_h.
\end{equation}
Since $F$ is transversal to the tangent space, radial variations detect
the Euler--Lagrange equation; compactly supported tangential variation
fields are generated by reparametrizations. Thus $\Delta_h\log\psi=0$
is precisely the stationarity condition \cite[Theorem~2.1]{LLS04}.

\subsection{Calabi compositions and completeness of centroaffine hypersurfaces}\label{sect:2.2}~

\vskip1mm
In this subsection, we first recall the Calabi composition introduced
in the centroaffine setting by Li--Wang \cite{LW91}, in the parametrized
form used in \cite{CHM17}.

\begin{definition}\label{def:2.1}
Let $\phi_i:M_i^{n_i}\to\mathbb R^{n_i+1}$ be non-degenerate centroaffine
hypersurfaces for $i=1,2$, and let $\lambda\in\mathbb{R}\setminus\{0,-1\}$.
Then the Calabi composition of $\phi_1$ and $\phi_2$, defined by
\begin{equation}\label{eqn:2.17}
\mathcal{C}_\lambda(t,p,q)
=(e^t\phi_1(p),e^{-\lambda t}\phi_2(q)):\mathbb{R}\times
M^{n_1}_1\times M^{n_2}_2\to\mathbb{R}^{n_1+n_2+2},
\end{equation}
is a non-degenerate centroaffine hypersurface, while the Calabi composition
of $\phi_1$ and a point, defined by
\begin{equation}\label{eqn:2.18}
{\widetilde{\mathcal C}_\lambda}(t,p)
=(e^t\phi_1(p),e^{-\lambda t}):\mathbb{R}\times
M^{n_1}_1\to\mathbb{R}^{n_1+2},
\end{equation}
is also a non-degenerate centroaffine hypersurface.
\end{definition}

In the hyperbolic setting considered below, we take $\lambda>0$.
The induced centroaffine metrics and the relevant completeness properties
of the two constructions will be established in Propositions
\ref{pro:4.1} and \ref{pro:5.1}, respectively.

Following \cite{LLS04}, we distinguish two notions of completeness.
A centroaffine hypersurface is \emph{centroaffinely complete} if its
centroaffine metric $h$ is complete, and \emph{Euclidean complete} if
the metric induced by a Euclidean inner product on the ambient space is
complete.  The latter condition is independent of the chosen inner product,
since all Euclidean norms on a finite-dimensional vector space are uniformly
equivalent, as are their pullback metrics.

Finally, we recall that every proper injective immersion into Euclidean
space is an embedding. We shall also use the fact that, for the
one-dimensional factors appearing in Section \ref{sect:4}, formulas
\eqref{eqn:2.13} and \eqref{eqn:2.14} remain valid with $n=1$ and with
$\Delta_h$ denoting the one-dimensional Laplace--Beltrami operator.

\begin{lemma}\label{lem:proper-complete}
Let $F:N^m\to\mathbb R^{m+1}$ be a locally strongly convex hyperbolic
centroaffine immersion, and let $g_F$ denote the metric induced by a
Euclidean inner product on $\mathbb R^{m+1}$. Then $F$ is proper if and
only if $g_F$ is complete.
\end{lemma}

\begin{proof}
Suppose first that $g_F$ is complete. Denote the centroaffine metric of
$F$ by $h_F$, the Euclidean second fundamental form by $B$, and the
Euclidean normal component of the position vector by $\eta$. Taking
normal components in the Gauss formula \eqref{eqn:2.1} with
$\varepsilon=-1$ gives $B(U,V)=h_F(U,V)\eta$. Hence the function
$E=\tfrac12|F|^2$ satisfies
\[
\operatorname{Hess}_{g_F}E(U,V)
=g_F(U,V)+\langle F,B(U,V)\rangle
=g_F(U,V)+|\eta|^2h_F(U,V),
\]
so that $\operatorname{Hess}_{g_F}E\ge g_F$. Fix $x_0\in N$ and write
$r=d_{g_F}(x_0,x)$. Integrating this inequality twice along a unit-speed
minimizing $g_F$-geodesic from $x_0$ to $x$ yields
\[
E(x)\ge E(x_0)
-|\operatorname{grad}_{g_F}E(x_0)|_{g_F}\,r+\tfrac12r^2.
\]
Therefore the inverse image under $F$ of a bounded set lies in a closed
$g_F$-ball, which is compact by the Hopf--Rinow theorem. The inverse image
of a compact set is closed, hence compact, so $F$ is proper.

Conversely, suppose that $F$ is proper. For a piecewise smooth curve
$\sigma$ joining $x,y\in N$, the Euclidean distance between $F(x)$ and
$F(y)$ is at most the Euclidean length of $F\circ\sigma$, which equals
$L_{g_F}(\sigma)$; hence $|F(x)-F(y)|\le d_{g_F}(x,y)$. If $(p_j)$ is a
$g_F$-Cauchy sequence, then $(F(p_j))$ is a Euclidean Cauchy sequence and
so lies in a compact closed ball $\bar B\subset\mathbb R^{m+1}$. Properness
makes $F^{-1}(\bar B)$ compact, so a subsequence $(p_{j_\ell})$ converges
in $N$ to some $p_\infty$. Since $d_{g_F}$ induces the manifold topology,
$d_{g_F}(p_{j_\ell},p_\infty)\to0$, and the Cauchy property then forces
$p_j\to p_\infty$. Thus $g_F$ is complete.
\end{proof}

\section{The two-dimensional case}\label{sect:3}

This section develops the two-dimensional example that initiates the first
family in the introduction. We first explain the geometric choices behind
its construction and then verify its curvature, properness, and two
contrasting completeness properties.

\subsection{Construction of the surface in Example \ref{exa:1.1}}\label{sect:3.1}~

\vskip1mm
For the present purpose, an ordinary Calabi composition is too rigid. Indeed,
composing a centroaffine curve with a point gives a product centroaffine metric
and hence a flat surface. Under the relevant completeness and extremality
assumptions, this construction exactly leads back to Wang's class. It therefore
cannot produce the desired example with positive Gauss curvature. We instead
allow the curve factor to vary along a transverse direction, leading to a
warped product construction.

We begin with a curve whose centroaffine structure
equations are explicit. More precisely, $\gamma(v)=Q(v/\sqrt2)$ satisfies
\begin{equation}\label{eqn:3.1}
\gamma''=\gamma-v\gamma',\ \
[\gamma,\gamma']=\tfrac1{\sqrt2}e^{-v^2/2}.
\end{equation}
The first identity controls the second derivatives in the curve direction,
while the determinant in the second keeps the immersion nondegenerate. Using
this curve, consider
\begin{equation}\label{eqn:3.2}
Z(u,v)=(a(u)\gamma(v),b(u)),\ \  a,b>0.
\end{equation}

We now choose the two functions governing the transverse motion. The
normalization $a'b-ab'=2b$ and the notation $w=-b'/b$ yield $a'+aw=2$.
Differentiating \eqref{eqn:3.2} and using \eqref{eqn:3.1}, we find
\begin{equation}\label{eqn:3.3}
Z_{uu}=-wZ_u-w'Z,\ \
Z_{uv}=\tfrac{a'}a Z_v,\ \
Z_{vv}=\tfrac a2Z_u-vZ_v+\tfrac{aw}2Z.
\end{equation}
Thus the centroaffine metric induced by the transversal $Z$ is $h=-w'\,du^2
+(aw/2)\,dv^2$. We impose $-w'=aw$ such that $h$ is conformal to the flat
metric $du^2+dv^2/2$. With
\begin{equation}\label{eqn:3.4}
\rho(u):=a(u)w(u),
\end{equation}
the metric becomes $h=\rho(du^2+dv^2/2)$. Moreover, $(a-w)'=2$, so a
translation of $u$ allows us to take $a=2u+w$. The remaining equations
are
\begin{equation}\label{eqn:3.5}
w'=-2uw-w^2,\ \  b'=-wb.
\end{equation}
The substitution $w=\Phi'/\Phi$ turns the Riccati equation into $\Phi''+2u
\Phi'=0$. Choosing the positive solution
\begin{equation}\label{eqn:3.6}
\Phi(u)=\int_{-\infty}^u e^{-s^2}\,ds,\ \
w=\tfrac{1}\Phi e^{-u^2},\ \
b=\tfrac1\Phi,\ \  a=2u+w,
\end{equation}
defines all the functions on $\mathbb R$ and gives $\Phi,w,b>0$. Furthermore,
\begin{equation}\label{eqn:3.7}
a\Phi=2u\Phi+e^{-u^2}
=2\int_{-\infty}^u(u-s)e^{-s^2}\,ds>0,
\end{equation}
so $a,\rho>0$. In particular, the determinant $[Z,Z_u,Z_v]=\sqrt2\,ab
e^{-v^2/2}$ does not vanish, and \eqref{eqn:3.2} defines a locally strongly
convex hyperbolic centroaffine surface.

The choice of the conformal factor is also tailored to the extremal equation.
Indeed,
\begin{equation}\label{eqn:3.8}
\psi=\tfrac{w^2}{4b^2}e^{v^2}
=\tfrac14 e^{v^2-2u^2},\ \
\Delta_h=\tfrac1\rho(\partial_u^2+2\partial_v^2).
\end{equation}
The $u$-term and $v$-term in $\log\psi$ contribute with the opposite signs and
cancel under $\Delta_h$. Hence $\Delta_h\log\psi=0$, and the resulting surface
is precisely $F_2$ from Example~\ref{exa:1.1}. The next subsection establishes
its global and curvature properties.

\subsection{Completion of the proof of Theorem \ref{thm:1.1}}\label{sect:3.2}~

\vskip1mm
Put $Z=F_2$. The positivity of $\Phi$ on $\mathbb R$ shows at once that
all functions in \eqref{eqn:1.2} and \eqref{eqn:1.3}, and hence $Z$
itself, are smooth. We verify the remaining assertions in six steps.

\smallskip
\emph{{\bf Step 1}: Show that $F_2$ is a locally strongly convex hyperbolic
centroaffine surface.}

\smallskip
The definitions in \eqref{eqn:1.2} immediately give
\begin{equation}\label{eqn:3.9}
\begin{aligned}
w'&=-aw,\ \ a'=2-aw,\ \ b'=-wb,\\
a''&=(a^2+aw-2)w,\ \ b''=(a+w)wb.
\end{aligned}
\end{equation}
The rescaled curve $\gamma(v)=Q(v/\sqrt2)$ satisfies
\begin{equation}\label{eqn:3.10}
\gamma'=(\sqrt2 f(\tfrac{v}{\sqrt2}),\tfrac1{\sqrt2}),\ \
\gamma''=(e^{-\frac{v^2}{2}},0)=\gamma-v\gamma',\ \
[\gamma,\gamma']=\tfrac1{\sqrt2}e^{-\frac{v^2}{2}}.
\end{equation}
Differentiation of \eqref{eqn:1.3} yields
\begin{equation}\label{eqn:3.11}
Z_u=(a'\gamma,b'),\ \
Z_v=(a\gamma',0).
\end{equation}
Combining these identities, we obtain
\begin{equation}\label{eqn:3.12}
[Z,Z_u,Z_v]
=a(a'b-ab')[\gamma,\gamma']
=\sqrt2\,ab\,e^{-\frac{v^2}{2}}.
\end{equation}
Recall that $\rho=aw$. By \eqref{eqn:3.7}, we have $a>0$. Together with
$w=e^{-u^2}/\Phi>0$, this also guarantees that $\rho>0$ on $\mathbb R$.

Since $b=1/\Phi>0$, the determinant in \eqref{eqn:3.12} never vanishes.
Thus $Z_u,Z_v$, and $Z$ are linearly independent everywhere: $Z$ is an
immersion, and its position vector is transversal to the tangent plane.
Therefore $F_2$ is a centroaffine surface.

A further differentiation of \eqref{eqn:3.11}, using \eqref{eqn:3.9} and
\eqref{eqn:3.10}, leads to
\begin{equation}\label{eqn:3.13}
\begin{aligned}
Z_{uu}&=-wZ_u+\rho Z,\ \ Z_{uv}=\tfrac{a'}{a}Z_v,\\
Z_{vv}&=\tfrac{a}{2}Z_u-vZ_v+\tfrac{\rho}{2}Z.
\end{aligned}
\end{equation}
With $\varepsilon=-1$, comparison of the coefficients of $Z$ in
\eqref{eqn:3.13} with the Gauss formula \eqref{eqn:2.1} identifies
the centroaffine metric as
\begin{equation}\label{eqn:3.14}
h=\rho\,du^2+\tfrac{\rho}{2}\,dv^2.
\end{equation}
The positivity of $\rho$ makes $h$ positive definite. With
$\varepsilon=-1$, this proves that $F_2$ is locally strongly convex and
hyperbolic.

\smallskip
\emph{{\bf Step 2}: Show that $F_2$ is extremal.}

\smallskip
We first compute the two quantities entering the Tchebychev function. From
\eqref{eqn:3.12} and \eqref{eqn:3.14}, they are
\begin{equation}\label{eqn:3.15}
\det(h_{ij})=\tfrac{\rho^2}{2},\ \
[Z_u,Z_v,Z]^2=2a^2b^2e^{-v^2}.
\end{equation}
The definition \eqref{eqn:2.13} takes the form
\begin{equation}\label{eqn:3.16}
\psi=\tfrac{\det(h_{ij})}{[Z_u,Z_v,Z]^2}
=\tfrac14 e^{v^2-2u^2},
\end{equation}
and hence
\begin{equation}\label{eqn:3.17}
\log\psi=-\log4+v^2-2u^2.
\end{equation}
For the metric \eqref{eqn:3.14}, $\sqrt{\det(h_{ij})}(h^{ij})=\operatorname
{diag}(1/\sqrt2,\sqrt2)$. The Laplace--Beltrami formula therefore implies that
\begin{equation}\label{eqn:3.18}
\Delta_h \log\psi=\tfrac1\rho\left((\log\psi)_{uu}+2(\log\psi)_{vv}\right)=0.
\end{equation}
By \eqref{eqn:2.15}, the vanishing of this Laplacian is exactly the
centroaffine extremal equation. This completes the proof of Step 2.

\smallskip
\emph{{\bf Step 3}: Show that the centroaffine metric $h$ has positive Gauss
curvature.}

\smallskip
We shall compute the Gauss curvature by using the difference tensor. The
tangential coefficients in \eqref{eqn:3.13} determine the centroaffine
connection:
\begin{equation}\label{eqn:3.19}
\nabla_{\partial_u}\partial_u=-w\partial_u,\ \
\nabla_{\partial_u}\partial_v=\tfrac{a'}{a}\partial_v,\ \
\nabla_{\partial_v}\partial_v=\tfrac{a}{2}\partial_u-v\partial_v.
\end{equation}
In the same coordinates, the Levi--Civita connection of the metric
\eqref{eqn:3.14} takes the form
\begin{equation}\label{eqn:3.20}
\widehat\nabla_{\partial_u}\partial_u=\tfrac{\rho'}{2\rho}\partial_u,\ \
\widehat\nabla_{\partial_u}\partial_v=\tfrac{\rho'}{2\rho}\partial_v,\ \
\widehat\nabla_{\partial_v}\partial_v=-\tfrac{\rho'}{4\rho}\partial_u.
\end{equation}
Thus the identity $K=\nabla-\widehat\nabla$ in \eqref{eqn:2.2} reads
\begin{align}\label{eqn:3.21}
\begin{split}
K(\partial_u,\partial_u)
&=-(w+\tfrac{\rho'}{2\rho})\partial_u,\\
K(\partial_u,\partial_v)
&=(\tfrac{a'}{a}-\tfrac{\rho'}{2\rho})\partial_v,\\
K(\partial_v,\partial_v)
&=(\tfrac{a}{2}+\tfrac{\rho'}{4\rho})\partial_u-v\partial_v.
\end{split}
\end{align}
Substituting \eqref{eqn:3.14} and \eqref{eqn:3.21} into the Gauss equation
\eqref{eqn:2.5} with $\varepsilon=-1$, we obtain
\begin{align}\label{eqn:3.22}
\begin{split}
\widehat{R}(\partial_u,\partial_v)\partial_v
&=-(h(\partial_v,\partial_v)\partial_u
-h(\partial_u,\partial_v)\partial_v)
-[K_{\partial_u},K_{\partial_v}]\partial_v\\
&=\tfrac{1}{2a^2}(a^2+2-\rho(a^2+1))\partial_u,
\end{split}
\end{align}
Here, the simplification uses $a'=2-\rho$ and $\rho'=\tfrac{\rho}{a}
(2-\rho-a^2)$. Since $\partial_u$ and $\partial_v$ are orthogonal for
$h$, the Gauss curvature is
\begin{equation}\label{eqn:3.23}
K_h=\tfrac{h(\widehat{R}(\partial_u,\partial_v)\partial_v,\partial_u)}
{h(\partial_u,\partial_u)h(\partial_v,\partial_v)}
=\tfrac{a^2+2-\rho(a^2+1)}{a^2\rho}.
\end{equation}
The denominator is positive, so it remains to show that $r(u):=a^2+2
-\rho(a^2+1)>0$ on $\mathbb R$. We first examine the possible zeros of
$r$. Differentiating and using the preceding identities for $a'$ and
$\rho'$, we find
\begin{equation}\label{eqn:3.24}
r'=2a(1-\rho)(2-\rho)
-\tfrac{\rho}{a}(2-\rho-a^2)(1+a^2).
\end{equation}
If $r(u_0)=0$, then, with $a=a(u_0)$,
\begin{equation}\label{eqn:3.25}
\rho(u_0)=\tfrac{a^2+2}{a^2+1}.
\end{equation}
At $u_0$, the expression for $r'$ consequently reduces to
\begin{equation}\label{eqn:3.26}
r'(u_0)=\tfrac{a^5(a^2+3)}{(a^2+1)^2}>0.
\end{equation}

We now determine the behaviour of $r$ at both ends of the real line. As
$u\to+\infty$,
\begin{equation}\label{eqn:3.27}
\Phi\to\sqrt\pi,\ \
w=\tfrac{e^{-u^2}}{\Phi}\to0^+,\ \
a=2u+w\to+\infty,\ \
\rho=aw\to0^+.
\end{equation}
It follows that
\begin{equation}\label{eqn:3.28}
\tfrac{r}{a^2}=1-\rho+\tfrac{2-\rho}{a^2}\to1,\ \
r\to+\infty,\ \ \text{as}\ \ u\to+\infty.
\end{equation}
As $u\to-\infty$, repeated integration by parts in the Gaussian tail
yields
\begin{equation}\label{eqn:3.29}
\Phi(u)=\tfrac{e^{-u^2}}{-2u}
\left(1-\tfrac{1}{2u^2}+\tfrac{3}{4u^4}-\tfrac{15}{8u^6}
+\tfrac{105}{16u^8}+O(u^{-10})\right).
\end{equation}
Using $w=e^{-u^2}/\Phi$, we then expand
\begin{equation}\label{eqn:3.30}
w=-2u-u^{-1}+u^{-3}-\tfrac52u^{-5}+\tfrac{37}{4}u^{-7}+O(u^{-9}).
\end{equation}
The relation $a=2u+w$ shows that
\begin{equation}\label{eqn:3.31}
a=-u^{-1}+u^{-3}-\tfrac52u^{-5}+\tfrac{37}{4}u^{-7}+O(u^{-9}),
\end{equation}
and the product $\rho=aw$ satisfies
\begin{equation}\label{eqn:3.32}
\rho=2-u^{-2}+3u^{-4}-\tfrac{25}{2}u^{-6}+O(u^{-8}).
\end{equation}
Substituting these expansions into the definition of $r$, we arrive at
\begin{equation}\label{eqn:3.33}
r=\tfrac32u^{-6}+O(u^{-8})\to0^+.
\end{equation}
In particular, $r$ is positive for all sufficiently negative $u$. If $r$
had a zero, its zero set would be nonempty, closed, and bounded below, and
hence would have a smallest element $u_0$. Continuity would then force $r>0$
on $(-\infty,u_0)$. This contradicts \eqref{eqn:3.26}, which requires $r<0$
immediately to the left of $u_0$. Therefore $r$ has no zero and remains
positive on $\mathbb R$. Since $a^2\rho>0$, formula \eqref{eqn:3.23}
proves that $K_h>0$.

\smallskip
\emph{{\bf Step 4}: Show that the centroaffine metric $h$ is incomplete.}

\smallskip
We examine the metric along the positive $u$-direction. Since
$\Phi(u)\to\sqrt\pi$, we first have
\begin{equation}\label{eqn:3.34}
w(u)=\tfrac{e^{-u^2}}{\Phi(u)}
\sim\tfrac{1}{\sqrt\pi}e^{-u^2}.
\end{equation}
Since $a(u)=2u+w(u)$ and $w(u)=o(u)$, we have $a(u)\sim2u$. Multiplying
the preceding two asymptotic relations yields
\begin{equation}\label{eqn:3.35}
\rho(u)=a(u)w(u)
\sim\tfrac{2u}{\sqrt\pi}e^{-u^2}.
\end{equation}
Consequently,
\begin{equation}\label{eqn:3.36}
\sqrt{\rho(u)}
\sim\left(\tfrac{2}{\sqrt\pi}\right)^{1/2}
u^{1/2}e^{-u^2/2}.
\end{equation}
The function on the right is integrable at infinity. Hence, for every
$u_0>0$,
\begin{equation}\label{eqn:3.37}
\int_{u_0}^{\infty}\sqrt{\rho(u)}\,du<\infty.
\end{equation}
Now fix $v_0\in\mathbb R$ and consider the ray $\sigma(u)=(u,v_0)$
for $u\ge u_0$. It leaves every compact subset of $\mathbb R^2$, but
its length is finite by \eqref{eqn:3.14} and the preceding estimate:
\begin{equation}\label{eqn:3.38}
L_h(\sigma)
=\int_{u_0}^{\infty}\sqrt{h(\partial_u,\partial_u)}\,du
=\int_{u_0}^{\infty}\sqrt{\rho(u)}\,du<\infty.
\end{equation}
Thus a divergent curve has finite $h$-length. The divergent-curve
criterion, or equivalently the Hopf--Rinow theorem, shows that $h$ is
incomplete.

\smallskip
\emph{{\bf Step 5}: Show that $F_2$ is a proper embedding.}

\smallskip
Let $q(\xi)=e^{-\xi^2}+2\xi f(\xi)$ denote the first component of
$Q(\xi)$. Its derivatives are
\begin{equation}\label{eqn:3.39}
q'(\xi)=2f(\xi),\ \ q''(\xi)=2e^{-\xi^2}>0.
\end{equation}
Because $q''>0$ and $q'(0)=2f(0)=0$, the function $q$ is strictly convex
and has its unique minimum at $0$. Since $q(0)=1$, we have $q(\xi)\ge1$.
The identity $\gamma(v)=(q(v/\sqrt2),v/\sqrt2)$ therefore implies
\begin{equation}\label{eqn:3.40}
|\gamma(v)|^2=q(v/\sqrt2)^2+\tfrac{v^2}{2}
\geq1+\tfrac{v^2}{2}.
\end{equation}

We begin with properness. Let $((u_j,v_j))$ be a sequence in $\mathbb R^2$
for which $(Z(u_j,v_j))$ is bounded, and choose $C>0$ such that
\begin{equation}\label{eqn:3.41}
|Z(u_j,v_j)|\leq C
\end{equation}
for all $j$. Since $Z(u,v)=(a(u)\gamma(v),b(u))$, we have
\begin{equation}\label{eqn:3.42}
b(u_j)\leq C,\ \  a(u_j)|\gamma(v_j)|\leq C.
\end{equation}
As $u\to-\infty$, $\Phi(u)\to0$ and $b(u)=1/\Phi(u)\to\infty$. The first
estimate in \eqref{eqn:3.42} therefore bounds $(u_j)$ from below. At the
positive end, \eqref{eqn:3.40} guarantees that $|\gamma(v_j)|\ge1$, while
$a(u)\sim2u\to\infty$. The second estimate bounds $(u_j)$ from above. Therefore,
all $u_j$ lie in a compact interval $J$. The positive continuous function $a$
has a minimum $\delta>0$ on $J$, so $|\gamma(v_j)|\le C/\delta$ by \eqref{eqn:3.42}.
Equation \eqref{eqn:3.40} then bounds $(v_j)$. Hence $((u_j,v_j))$ is bounded
in $\mathbb R^2$.

Now let $\Omega\subset\mathbb R^3$ be compact. The preceding argument
shows that $Z^{-1}(\Omega)$ is bounded; it is closed because $Z$ is
continuous and $\Omega$ is closed. The Heine--Borel theorem
\cite[Theorem 2.41]{Rud76} makes $Z^{-1}(\Omega)$ compact. Thus $Z$ is
proper.

For injectivity, write $Z=(Z_1,Z_2,Z_3)$. Since $b'=-wb<0$, the third
component $Z_3=b(u)$ determines $u$ uniquely. With $u$ fixed,
$Z_2=a(u)v/\sqrt2$ and $a(u)>0$ determine $v$. Hence $Z$ is injective.
Together with the immersion property from Step~1 and properness just
proved, this shows that $F_2$ is a proper embedding.

\smallskip
\emph{{\bf Step 6}: Show that $F_2$ is Euclidean complete.}

\smallskip
Steps 1 and 5 show that $F_2$ is a proper locally strongly convex
hyperbolic centroaffine immersion, so Lemma \ref{lem:proper-complete}
yields Euclidean completeness.\qed

\section{A $3$-dimensional complete centroaffine extremal hypersurface}\label{sect:4}

\subsection{Calabi composition with two factors}\label{sect:4.1}~

\vskip1mm
The metric, Tchebychev function, and Laplacian formulas below are due to
Li--Li--Simon \cite{LLS04}; see also \cite[Proposition 3.1]{CHM17}
for the metric and flatness. We shall apply them to factors that need not be
extremal, allowing their nonzero constant contributions to cancel, and
establish the global properties needed for our examples.

\begin{proposition}\label{pro:4.1}
Let $\phi_i:M_i^{n_i}\to\mathbb R^{n_i+1}$ be locally strongly convex
hyperbolic centroaffine hypersurfaces for $i=1,2$ and $n_1,n_2\ge1$.
Denote the corresponding  centroaffine metrics and Tchebychev functions
by $h_1,h_2$ and $\psi_1,\psi_2$, respectively. Then, for $\lambda>0$,
the Calabi composition with two factors
\begin{equation}\label{eqn:4.1}
\mathcal{C}_\lambda(t,p,q)
=(e^t\phi_1(p),e^{-\lambda t}\phi_2(q)):\mathbb{R}\times
M^{n_1}_1\times M^{n_2}_2\to\mathbb{R}^{n_1+n_2+2}
\end{equation}
is a locally strongly convex hyperbolic centroaffine hypersurface with
centroaffine metric
\begin{equation}\label{eqn:4.2}
h_{\mathcal C_\lambda}
=\lambda\,dt^2\oplus\tfrac{\lambda}{\lambda+1}h_1
\oplus\tfrac{1}{\lambda+1}h_2.
\end{equation}
Moreover, the following properties hold:
\begin{enumerate}
\item[$(1)$]
$h_{\mathcal C_\lambda}$ is flat if and only if both $h_1$ and $h_2$ are
flat.

\item[$(2)$]
$h_{\mathcal C_\lambda}$ is complete if and only if both $h_1$ and $h_2$
are complete.

\item[$(3)$]
The Tchebychev function is
\begin{equation}\label{eqn:4.3}
\psi_{\mathcal C_\lambda}
=\tfrac{\lambda^{n_1+1}}{(\lambda+1)^{n_1+n_2+2}}
e^{-2\{(n_1+1)-\lambda(n_2+1)\}t}\psi_1\psi_2.
\end{equation}

\item[$(4)$]
Its Laplacian satisfies
\begin{equation}\label{eqn:4.4}
\Delta_{h_{\mathcal C_\lambda}}\log\psi_{\mathcal C_\lambda}
=\tfrac{\lambda+1}{\lambda}\Delta_{h_1}\log\psi_1
+(\lambda+1)\Delta_{h_2}\log\psi_2.
\end{equation}

\item[$(5)$]
$\mathcal C_\lambda$ is centroaffine extremal if and only if there exists
a constant $c\in\mathbb{R}$ such that
\begin{equation}\label{eqn:4.5}
\Delta_{h_1}\log\psi_1=\lambda c,\ \
\Delta_{h_2}\log\psi_2=-c.
\end{equation}

\item[$(6)$]
If $\mathcal C_\lambda$ is an embedding, then both $\phi_1$ and $\phi_2$
are embeddings.

\item[$(7)$]
$\mathcal C_\lambda$ is proper if and only if both $\phi_1$ and $\phi_2$
are proper.

\item[$(8)$]
The Euclidean metric $g_{\mathcal C_\lambda}$ induced by $\mathcal C_\lambda$
is complete if and only if the Euclidean metrics $g_1$ and $g_2$ induced by
$\phi_1$ and $\phi_2$, respectively, are both complete.
\end{enumerate}
\end{proposition}

\begin{proof}
For convenience, we write $Z=\mathcal C_\lambda$, and choose local
coordinates $(p^1,\ldots,p^{n_1})$ and $(q^1,\ldots,q^{n_2})$ on $M_1$
and $M_2$, respectively. Setting $\phi_{1,i}=\partial\phi_1/\partial p^i$
and $\phi_{2,\alpha}=\partial\phi_2/\partial q^\alpha$, by differentiating
\eqref{eqn:4.1} we obtain
\begin{equation}\label{eqn:4.6}
Z_t=(e^t\phi_1,-\lambda e^{-\lambda t}\phi_2),\ \
Z_i=(e^t\phi_{1,i},0),\ \
Z_{\bar\alpha}=(0,e^{-\lambda t}\phi_{2,\alpha}),
\end{equation}
and it then follows that
\begin{equation}\label{eqn:4.7}
\begin{split}
&[Z_t,Z_1,\ldots,Z_{n_1},Z_{\bar1},\ldots,Z_{\bar n_2},Z]\\
&=(-1)^{n_1}(\lambda+1)e^{\{(n_1+1)-\lambda(n_2+1)\}t}
[\phi_{1,1},\ldots,\phi_{1,n_1},\phi_1]
[\phi_{2,1},\ldots,\phi_{2,n_2},\phi_2]\ne0.
\end{split}
\end{equation}
This implies that $\mathcal{C}_\lambda:\mathbb{R}\times M^{n_1}_1\times
M^{n_2}_2\to\mathbb{R}^{n_1+n_2+2}$ is a centroaffine hypersurface.

Moreover, according to the following identities
\begin{equation}\label{eqn:4.8}
(e^t\phi_1,0)=\tfrac1{\lambda+1}Z_t+\tfrac\lambda{\lambda+1}Z,\ \
(0,e^{-\lambda t}\phi_2)=-\tfrac1{\lambda+1}Z_t+\tfrac1{\lambda+1}Z,
\end{equation}
we can apply the Gauss formula to calculate that
\begin{equation}\label{eqn:4.9}
\begin{aligned}
D_{\partial_t}Z_t&=(1-\lambda)Z_t+\lambda Z,\ \
D_{\partial_t}Z_i=D_{\partial_i}Z_t=Z_i,\\
D_{\partial_t}Z_{\bar\alpha}
&=D_{\partial_{\bar\alpha}}Z_t=-\lambda Z_{\bar\alpha},\ \
D_{\partial_i}Z_{\bar\alpha}=D_{\partial_{\bar\alpha}}Z_i=0,\\
D_{\partial_i}Z_j
&=Z_*\bigl(\nabla^1_{\partial_i}\partial_j\bigr)
+\tfrac{(h_1)_{ij}}{\lambda+1}Z_t
+\tfrac{\lambda(h_1)_{ij}}{\lambda+1}Z,\\
D_{\partial_{\bar\alpha}}Z_{\bar\beta}
&=Z_*\bigl(\nabla^2_{\partial_\alpha}\partial_\beta\bigr)
-\tfrac{(h_2)_{\alpha\beta}}{\lambda+1}Z_t
+\tfrac{(h_2)_{\alpha\beta}}{\lambda+1}Z,
\end{aligned}
\end{equation}
where $\partial_i:=\partial/\partial p^i$ and $\partial_\alpha:=\partial
/\partial q^\alpha$ are respectively the coordinate vector fields on $M_1$
and $M_2$. Here we also use the same notation $\partial_i$ for its lift to
the product and write $\partial_{\bar\alpha}$ for the lift of $\partial_\alpha$.
In particular, $\nabla^a$ denotes the centroaffine connection induced by
$\phi_a$, and $(h_1)_{ij}:=h_1(\partial_i,\partial_j)$ and $(h_2)_{\alpha
\beta}:=h_2(\partial_\alpha,\partial_\beta)$, where $1\le i,j\le n_1$ and
$1\le\alpha,\beta\le n_2$. Comparing these with \eqref{eqn:2.1}, we prove
\eqref{eqn:4.2}. Since $\lambda>0$, the centroaffine metric $h_{\mathcal
C_\lambda}$ is positive definite, and hence $\mathcal{C}_\lambda$ is
locally strongly convex and hyperbolic.

It also follows from \eqref{eqn:4.2} that the Levi--Civita connection is
the product connection, and hence its curvature is the direct sum of
the factor curvatures, proving (1). Together with \eqref{eqn:4.2} and the
completeness of $(\mathbb{R},\lambda\,dt^2)$, the assertion (2) follows
immediately. Using \eqref{eqn:4.2} again, we easily get
\begin{equation}\label{eqn:4.10}
\det(h_Z)=\tfrac{\lambda^{n_1+1}}
{(\lambda+1)^{n_1+n_2}}\det(h_1)\det(h_2).
\end{equation}
Combining this with the transversal determinant and \eqref{eqn:2.13},
we obtain \eqref{eqn:4.3} which proves (3). Furthermore, it is seen from
\eqref{eqn:4.2} that
\begin{equation}\label{eqn:4.11}
\Delta_{h_Z}=\tfrac1\lambda\partial_t^2
+\tfrac{\lambda+1}{\lambda}\Delta_{h_1}
+(\lambda+1)\Delta_{h_2}.
\end{equation}
The $t$-dependent term in $\log\psi_Z$ is linear, so applying this operator
to \eqref{eqn:4.3} proves \eqref{eqn:4.4}, and hence (4). Next, we set $d_1:
=\Delta_{h_1}\log\psi_1$ and $d_2:=\Delta_{h_2}\log\psi_2$. By means of
\eqref{eqn:2.15} and \eqref{eqn:4.4}, $Z$ is extremal precisely when $\lambda^{-1}
d_1(p)+d_2(q)=0$ on $M_1\times M_2$. Fixing one factor variable at a time
shows that $d_1$ and $d_2$ are constant. Consequently, $d_1=\lambda c$
and $d_2=-c$ for some $c\in\mathbb R$. The converse follows immediately from
\eqref{eqn:4.4}, proving (5).

Fix $p_0\in M_1$ and $q_0\in M_2$, and consider the closed embeddings
\begin{equation}\label{eqn:4.12}
\sigma_1(p)=(0,p,q_0),\ \ \sigma_2(q)=(0,p_0,q)
\end{equation}
of $M_1$ and $M_2$ into $\mathbb R\times M_1\times M_2$, respectively. Their
compositions with $Z$ are
\begin{equation}\label{eqn:4.13}
(Z\circ\sigma_1)(p)=(\phi_1(p),\phi_2(q_0)),\ \
(Z\circ\sigma_2)(q)=(\phi_1(p_0),\phi_2(q)).
\end{equation}
If $Z$ is an embedding, so are these two compositions. Identifying each
slice of the target with the corresponding Euclidean factor, we conclude
that $\phi_1$ and $\phi_2$ are embeddings. This proves (6).

To prove (7), we suppose first that both factors are proper. Since their
images avoid the origin, properness gives constants $\delta_i>0$ such that
$|\phi_i|\ge\delta_i$ on $M_i$, for $i=1,2$. Indeed, a sequence of image
points tending to zero would have a convergent parameter subsequence,
contradicting centroaffine transversality. Let $\Lambda\subset\mathbb
R^{n_1+n_2+2}$ be compact, and choose $C>\max\{1,\delta_1,\delta_2\}$
such that $|Z|\le C$ on $Z^{-1}(\Lambda)$. It follows that
\begin{equation}\label{eqn:4.14}
\delta_1e^t\le C,\ \ \delta_2e^{-\lambda t}\le C,
\end{equation}
and hence
\begin{equation}\label{eqn:4.15}
a:=\tfrac1\lambda\log\tfrac{\delta_2}{C}
\le t\le\log\tfrac{C}{\delta_1}=:b.
\end{equation}
Consequently,
\begin{equation}\label{eqn:4.16}
|\phi_1(p)|\le Ce^{-a},\ \ |\phi_2(q)|\le Ce^{\lambda b}.
\end{equation}
Let $\bar B_1\subset\mathbb R^{n_1+1}$ and $\bar B_2\subset\mathbb R^{n_2+1}$
be the closed balls centred at the origin with these respective radii. Then
\begin{equation}\label{eqn:4.17}
Z^{-1}(\Lambda)\subset
[a,b]\times\phi_1^{-1}(\bar B_1)\times\phi_2^{-1}(\bar B_2).
\end{equation}
The product on the right is compact by the properness of the factors. Since
$Z^{-1}(\Lambda)$ is closed, it is compact as well, and $Z$ is proper. Conversely,
if $Z$ is proper, its restrictions to the two closed slices above are proper.
These restrictions are precisely $\phi_1$ and $\phi_2$ with constant coordinate
blocks attached, so both factors are proper. This proves (7).

Finally, (8) follows from Lemma \ref{lem:proper-complete} applied to
$\phi_1$, $\phi_2$, and $Z$, together with (7).
\end{proof}

\begin{remark}\label{rem:4.1}
If one radial projection $\phi_i/|\phi_i|$ is injective, Proposition
\ref{pro:4.1} (6) becomes an equivalence. Indeed, centroaffine transversality
makes this projection a diffeomorphism onto its image. When both factors are
embeddings, the direction and norm of the corresponding coordinate block
determine its parameter and $t$, and the other parameter is then recovered
continuously.
\end{remark}

\subsection{Centroaffine plane curves with constant centroaffine mean curvature}\label{sect:4.2}~

\vskip1mm
To apply Proposition~\ref{pro:4.1}, we need two centroaffine plane curves
whose values of $\Delta_h\log\psi$ are nonzero constants of opposite sign.
We first classify the three constant
values used below. Put
\begin{equation}\label{eqn:4.20}
f(\xi)=\int_0^\xi e^{-r^2}\,dr.
\end{equation}
The family with $\Delta_h\log\psi=0$ in the proposition below is the
one-dimensional specialization of Wang's class~\cite[Example~3.3]{LLS04}.

\begin{proposition}\label{pro:4.2}
Let $F:I\to\mathbb{R}^2$ be a locally strongly convex hyperbolic
centroaffine curve on an interval $I\subset\mathbb{R}$ satisfying
\begin{equation}\label{eqn:4.21}
\Delta_h\log\psi=c
\end{equation}
for some constant $c$. Then
\begin{enumerate}
\item[$(1)$]
If $c=0$, then $F$ is locally centroaffinely equivalent to
\begin{equation}\label{eqn:4.22}
W_\alpha(w)=(e^w,e^{-\alpha w}),\ \  0<\alpha\leq1,
\end{equation}
whose image is a member of Wang's class:
\begin{equation}\label{eqn:4.23}
x^{\alpha_1}_1x^{\alpha_2}_2=1,\ \ \alpha_1,\alpha_2>0.
\end{equation}

\item[$(2)$]
If $c=-2$, then $F$ is locally centroaffinely equivalent to
\begin{equation}\label{eqn:4.24}
P(u)=(e^{u^2},e^{u^2}f(u)).
\end{equation}

\item[$(3)$]
If $c=2$, then $F$ is locally centroaffinely equivalent to
\begin{equation}\label{eqn:4.25}
Q(v)=(e^{-v^2}+2vf(v),v).
\end{equation}
\end{enumerate}
Here each example is defined on $\mathbb{R}$, is hyperbolic, and is
centroaffinely complete.
\end{proposition}

\begin{proof}
After restricting $I\subset\mathbb{R}$ if necessary, we choose a
centroaffine arc-length parameter $s$ such that $h=ds^2$. Then the
Gauss formula \eqref{eqn:2.1} becomes
\begin{equation}\label{eqn:4.26}
F''=a(s)F'+F
\end{equation}
for a smooth function $a$. Set $\omega=[F',F]$. Centroaffine
transversality implies that $\omega\ne0$, and \eqref{eqn:4.26}
yields
\begin{equation}\label{eqn:4.27}
\omega'=[F'',F]=a[F',F]=a\omega.
\end{equation}
For such a plane curve, by \eqref{eqn:2.13} we easily get $\psi=\omega^{-2}$.
Consequently,
\begin{equation}\label{eqn:4.28}
(\log\psi)'=-2a,\ \
\Delta_h\log\psi=(\log\psi)''=-2a'.
\end{equation}
Combining \eqref{eqn:4.21} and \eqref{eqn:4.28}, we obtain
\begin{equation}\label{eqn:4.29}
a(s)=a_0-\tfrac{1}{2}cs
\end{equation}
for some constant $a_0$. Since the Levi--Civita connection of $ds^2$
vanishes in this coordinate, the tangential part of \eqref{eqn:4.26}
shows that $K(\partial_s,\partial_s)=a\partial_s$. Thus \eqref{eqn:2.4}
gives
\begin{equation}\label{eqn:4.30}
T=a\partial_s.
\end{equation}

Suppose first that $c=0$. Then $a=a_0$ is constant. Reversing the
arc-length parameter replaces $a_0$ by $-a_0$, so we may assume that
$a_0=\mu\geq0$. Set
\begin{equation}\label{eqn:4.31}
\delta=\sqrt{\mu^2+4},\ \
p=\tfrac{1}{2}(\delta-\mu),\ \
q=\tfrac{1}{2}(\delta+\mu),\ \
\alpha=\tfrac{p}{q}=p^2\in(0,1].
\end{equation}
Then $pq=1$, $q-p=\mu$, and the roots of the characteristic equation
$r^2-\mu r-1=0$ are $q$ and $-p$. Since both components of $F$ solve
this Gauss equation, there exist constant vectors $A,B\in\mathbb R^2$
such that
\begin{equation}\label{eqn:4.32}
F(s)=Ae^{qs}+Be^{-ps}.
\end{equation}
If $A$ and $B$ were linearly dependent, then both $F$ and $F'$ take
values in a fixed line, which contradicts centroaffine transversality,
i.e., $[F',F]\ne0$. Hence $A$ and $B$ are linearly independent and there
exists $L\in\operatorname{GL}(2,\mathbb{R})$ such that $L(A)=(1,0)$ and
$L(B)=(0,1)$. Applying $L$ to \eqref{eqn:4.32}, we may therefore write
$F(s)=(e^{qs},e^{-ps})$. With $w=qs$ and $\alpha=p/q$, this becomes
$W_\alpha(w)$ in \eqref{eqn:4.22}. Moreover, writing $x_1=e^{qs}$ and
$x_2=e^{-ps}$, we have $x_1^p x_2^q=e^{pqs}e^{-pqs}=1$. Since $\delta
>\mu\geq0$, we have $p>0$. Together with $pq=1$ and $q-p=\mu\geq0$,
this yields $0<p\leq1\leq q$. Hence this is precisely Wang's equation
$x_1^{\alpha_1}x_2^{\alpha_2}=1$ with $(\alpha_1,\alpha_2)=(p,q)\in(0,1]
\times[1,\infty)$. To verify that this normalized range represents the
full Wang's class, we take arbitrary $\alpha_1,\alpha_2>0$. Applying
the centroaffine transformation $(x_1,x_2)\mapsto(x_2,x_1)$ if
necessary, we may assume that $\alpha_1\leq\alpha_2$. Set
\begin{equation}\label{eqn:4.33}
\rho=\sqrt{\alpha_1\alpha_2},\ \
\widehat p=\tfrac{\alpha_1}{\rho},\ \
\widehat q=\tfrac{\alpha_2}{\rho}.
\end{equation}
It follows that
\begin{equation}\label{eqn:4.34}
\widehat p\widehat q=1,\ \
0<\widehat p\leq1\leq\widehat q,\ \
x_1^{\alpha_1}x_2^{\alpha_2}
=(x_1^{\widehat p}x_2^{\widehat q})^\rho.
\end{equation}
Therefore, the above range merely normalizes the exponents and imposes
no restriction on Wang's class. Additionally, \eqref{eqn:4.30} and
\eqref{eqn:4.31} give $\|T\|_h=\mu=(1-\alpha)/\sqrt\alpha$. Since $h$
and $T$ are centroaffine invariants and the function $\alpha\mapsto(1
-\alpha)/\sqrt\alpha$ is strictly decreasing on $(0,1]$, the normalized
parameter $\alpha$ is uniquely determined by the invariant $\|T\|_h$.
In particular, $T=0$ if and only if $\alpha=1$.

Suppose next that $c=-2$. Then $a(s)=a_0+s$, and \eqref{eqn:4.26} is
$F''-(s+a_0)F'-F=0$. Set $t=(s+a_0)/\sqrt2$ and $\widetilde F(t)=F
(\sqrt2t-a_0)$. Similarly, it follows that
\begin{equation}\label{eqn:4.35}
\begin{aligned}
\widetilde F''-2t\widetilde F'-2\widetilde F=0,\ \
\widetilde F(t)=Ae^{t^2}+Be^{t^2}f(t),\ \ A,B\in\mathbb R^2.
\end{aligned}
\end{equation}
Obviously, $e^{t^2}$ and $e^{t^2}f(t)$ form a fundamental system: both
solve the Gauss equation in \eqref{eqn:4.35}, and their Wronskian is
$e^{t^2}\ne0$. Moreover, centroaffine transversality forces $A$ and $B$
to be linearly independent. Consequently, there exists a centroaffine
transformation $L\in\operatorname{GL}(2,\mathbb{R})$ such that $(L\circ
\widetilde F)(t)=(e^{t^2},e^{t^2}f(t))=P(t)$. Based on this, it holds
that
\begin{equation}\label{eqn:4.36}
\begin{aligned}
P'(u)&=(2ue^{u^2},1+2ue^{u^2}f(u)),\\
P''(u)&=2uP'(u)+2P(u),\ \ [P,P']=e^{u^2}.
\end{aligned}
\end{equation}
This verifies centroaffine transversality and proves the assertion (2).

Suppose now that $c=2$. Then $a(s)=a_0-s$, and \eqref{eqn:4.26} is $F''
+(s-a_0)F'-F=0$. By introducing $t=(s-a_0)/\sqrt2$ and $\widetilde F(t)=F
(a_0+\sqrt2t)$, we write the Gauss equation and its general vector-valued
solution as follows:
\begin{equation}\label{eqn:4.37}
\begin{aligned}
\widetilde F''+2t\widetilde F'-2\widetilde F=0,\ \
\widetilde F(t)=At+B(e^{-t^2}+2tf(t)),\ \
A,B\in\mathbb R^2.
\end{aligned}
\end{equation}
The scalar functions $t$ and $e^{-t^2}+2tf(t)$ form a fundamental system of
the Gauss equation in \eqref{eqn:4.37}, because their Wronskian is $-e^{-t^2}
\ne0$. The condition $[\widetilde F',\widetilde F]\ne0$ also implies that $A$
and $B$ are linearly independent. Up to a centroaffine transformation $L\in
\operatorname{GL}(2,\mathbb{R})$, we obtain $(L\circ\widetilde F)(t)=(e^{-t^2}
+2tf(t),t)=Q(t)$. In this case, using the fact $f'(v)=e^{-v^2}$, we further
have
\begin{equation}\label{eqn:4.38}
\begin{aligned}
Q'(v)&=(2f(v),1),\ \ [Q,Q']=e^{-v^2},\\
Q''(v)&=(2e^{-v^2},0)=-2vQ'(v)+2Q(v).
\end{aligned}
\end{equation}
This also verifies centroaffine transversality and proves the assertion (3).

It remains to establish the stated global properties of the three models.
For the Wang curve in \eqref{eqn:4.22}, direct differentiation shows that
\begin{equation}\label{eqn:4.39}
\begin{aligned}
&W_\alpha''=(1-\alpha)W_\alpha'+\alpha W_\alpha,\\
&[W_\alpha,W_\alpha']=-(1+\alpha)e^{(1-\alpha)w}.
\end{aligned}
\end{equation}
Comparing the first identity with the Gauss formula and then applying
\eqref{eqn:2.13}, we obtain
\begin{equation}\label{eqn:4.40}
h_{W_\alpha}=\alpha\,dw^2,\ \
\psi_{W_\alpha}=\tfrac{\alpha}{(1+\alpha)^2}e^{-2(1-\alpha)w},\ \
\Delta_{h_{W_\alpha}}\log\psi_{W_\alpha}=0,
\end{equation}
where the last equality follows from the fact $\Delta_{\alpha\,dw^2}
=\alpha^{-1}d^2/dw^2$.

For the curve $P$, the identities in \eqref{eqn:4.36} and the Gauss formula
yield $h_P=2\,du^2$. Since $[P,P']=e^{u^2}$, we deduce by using \eqref{eqn:2.13}
that
\begin{equation}\label{eqn:4.41}
\psi_P=2e^{-2u^2},\ \
\Delta_{h_P}\log\psi_P
=\tfrac{1}{2}\tfrac{d^2}{du^2}(\log2-2u^2)=-2.
\end{equation}
Similarly, it follows from \eqref{eqn:4.38} that $h_Q=2\,dv^2$ and $[Q,Q']
=e^{-v^2}$. Then there hold
\begin{equation}\label{eqn:4.42}
\psi_Q=2e^{2v^2},\ \
\Delta_{h_Q}\log\psi_Q
=\tfrac{1}{2}\tfrac{d^2}{dv^2}(\log2+2v^2)=2.
\end{equation}
Hence the three metrics are positive definite, so the corresponding curves
are hyperbolic. Moreover, each metric is a constant positive multiple of
the Euclidean metric on $\mathbb R$ and is therefore complete. This proves
all the global assertions.
\end{proof}

\subsection{Completion of the proof of Theorem \ref{thm:1.2}}\label{sect:4.3}~

\vskip1mm
Write $Z=F_3$. We establish the stated properties in the following four
steps.

\smallskip
\emph{{\bf Step 1}: Show that $F_3$ is a flat hyperbolic centroaffine
hypersurface.}

\smallskip
Differentiating \eqref{eqn:1.4} gives
\begin{equation}\label{eqn:4.43}
Z_t=(e^tP,-e^{-t}Q),\ \ Z_u=(e^tP',0),\ \ Z_v=(0,e^{-t}Q'),
\end{equation}
which together with \eqref{eqn:4.36} and \eqref{eqn:4.38} yield
\begin{equation}\label{eqn:4.44}
[P,P']=e^{u^2},\ \ [Q,Q']=e^{-v^2},
\end{equation}
where $[\cdot,\cdot]$ denotes the determinant of the indicated ambient
vectors. It then follows that
\begin{equation}\label{eqn:4.45}
[Z_t,Z_u,Z_v,Z]
=-2[P,P'][Q,Q']
=-2e^{u^2-v^2}\ne0.
\end{equation}
Consequently, $dF_3$ has rank three and the position vector is nowhere
tangent to the image of $dF_3$. Hence $F_3$ is a centroaffine hypersurface.

Moreover, using $P''=2uP'+2P$ and $Q''=-2vQ'+2Q$, by calculation we obtain
\begin{equation}\label{eqn:4.46}
\begin{aligned}
Z_{tt}&=Z,\ \ Z_{tu}=Z_u,\ \ Z_{tv}=-Z_v,\ \ Z_{uv}=0,\\
Z_{uu}&=Z_t+2uZ_u+Z,\ \ Z_{vv}=-Z_t-2vZ_v+Z.
\end{aligned}
\end{equation}
Comparing the coefficient of the position vector in \eqref{eqn:4.46} with
the Gauss formula \eqref{eqn:2.1}, we further have
\begin{equation}\label{eqn:4.47}
h=dt^2+du^2+dv^2,
\end{equation}
and hence the positive definite metric $h$ is flat. This implies that the
hypersurface $F_3$ is locally strongly convex and hyperbolic.

\smallskip
\emph{{\bf Step 2}: Show that $F_3$ is centroaffine extremal and
centroaffine complete.}

\smallskip
Applying \eqref{eqn:4.45} and \eqref{eqn:4.47}, we easily get
\begin{equation}\label{eqn:4.48}
\det(h_{ij})=1,\ \
[Z_t,Z_u,Z_v,Z]^2=4e^{2u^2-2v^2}.
\end{equation}
Substituting these expressions into \eqref{eqn:2.13} yields
\begin{equation}\label{eqn:4.49}
\psi=\tfrac{\det(h_{ij})}{[Z_t,Z_u,Z_v,Z]^2}
=\tfrac{1}{4}e^{-2u^2+2v^2},
\end{equation}
and hence
\begin{equation}\label{eqn:4.50}
\log\psi=-\log4-2u^2+2v^2,\ \
\Delta_h\log\psi=0.
\end{equation}
Together with \eqref{eqn:2.15}, we see that $F_3$ is centroaffine extremal.
Notice from \eqref{eqn:4.47} that $h$ is the standard Euclidean metric on
$\mathbb{R}^3$. Therefore, the centroaffine metric $h$ is complete.

Furthermore, comparison of \eqref{eqn:4.46} with the Gauss formula
immediately gives
\begin{equation}\label{eqn:4.51}
K(\partial_t,\partial_t)=0,\ \
K(\partial_u,\partial_u)=\partial_t+2u\partial_u,\ \
K(\partial_v,\partial_v)=-\partial_t-2v\partial_v,
\end{equation}
where we used the fact that the metric $h$ in \eqref{eqn:4.47} has vanishing
Christoffel symbols. By the total symmetry in \eqref{eqn:2.3}, we take the
$h$-trace in \eqref{eqn:2.4} to obtain
\begin{equation}\label{eqn:4.52}
T=\tfrac{2}{3}(u\partial_u-v\partial_v),
\end{equation}
and then it can be checked from \eqref{eqn:2.12} that
\begin{equation}\label{eqn:4.53}
\mathcal{T}(\partial_t)=0,\ \
\mathcal{T}(\partial_u)=\tfrac{2}{3}\partial_u,\ \
\mathcal{T}(\partial_v)=-\tfrac{2}{3}\partial_v.
\end{equation}

\smallskip
\emph{{\bf Step 3}: Show that $F_3$ is a proper embedding.}

\smallskip
Denote by $q(v)=e^{-v^2}+2vf(v)$ the first component of $Q(v)$ in
\eqref{eqn:4.25}. Since $f(v)$ has the same sign as $v$, and
\begin{equation}\label{eqn:4.54}
q'(v)=2f(v),\ \ q'(0)=0,\ \ q''(v)=2e^{-v^2}>0,
\end{equation}
the function $q(v)$ has a unique global minimum at $v=0$. Moreover,
$f(v)\to\pm\sqrt\pi/2$ as $v\to\pm\infty$. It follows that
\begin{equation}\label{eqn:4.55}
q(v)\ge q(0)=1,\ \ q(v)\to\infty\ \
\text{as}\ \ |v|\to\infty.
\end{equation}
Let now $\Sigma\subset\mathbb R$ be compact. Since $\Sigma$ is
bounded, by \eqref{eqn:4.55} the inverse image $q^{-1}(\Sigma)$
is bounded. It is also closed because $q$ is continuous and
$\Sigma$ is closed. It follows that $q^{-1}(\Sigma)$ is compact
by the Heine--Borel theorem and $q:
\mathbb R\to\mathbb R$ is therefore proper. Equivalently, if
a sequence $(q(v_j))$ is bounded, then $(v_j)$ is bounded.

Write $Z=(Z_1,Z_2,Z_3,Z_4)$. Then the first and third components
in \eqref{eqn:1.4} are positive and satisfy
\begin{equation}\label{eqn:4.56}
Z_1=e^{t+u^2},\ \ Z_3=e^{-t}q(v),\ \ Z_1Z_3=e^{u^2}q(v)\ge1.
\end{equation}
Consequently,
\begin{equation}\label{eqn:4.57}
|Z|^2\ge Z_1^2+Z_3^2\ge2Z_1Z_3=2e^{u^2}q(v)\ge2.
\end{equation}
It is clear that the image $Z(\mathbb R^3)$ stays at Euclidean
distance at least $\sqrt2$ from the origin.

Then, we prove that $Z$ is proper. Suppose that $|Z(t_j,u_j,v_j)|$
is bounded for a sequence $((t_j,u_j,v_j))$ in $\mathbb{R}^3$.
Choose $C>0$ such that
\begin{equation}\label{eqn:4.58}
|Z(t_j,u_j,v_j)|\le C
\end{equation}
for all $j$. The preceding estimate in \eqref{eqn:4.57} gives
\begin{equation}\label{eqn:4.59}
2e^{u_j^2}q(v_j)
\le |Z(t_j,u_j,v_j)|^2\le C^2,
\end{equation}
and hence it follows from \eqref{eqn:4.55} that
\begin{equation}\label{eqn:4.60}
e^{u_j^2}\le e^{u_j^2}q(v_j)\le\tfrac{C^2}{2}.
\end{equation}
Thus $(u_j)$ is bounded. On the other hand, it is also known from
\eqref{eqn:4.59} that
\begin{equation}\label{eqn:4.61}
q(v_j)\le e^{u_j^2}q(v_j)\le\tfrac{C^2}{2}.
\end{equation}
Since $q(v)\to\infty$ as $|v|\to\infty$, the sequence $(v_j)$
is bounded.

Next, we control the remaining parameter. The first component
satisfies
\begin{equation}\label{eqn:4.62}
Z_1(t_j,u_j,v_j)=e^{t_j+u_j^2}\le C,
\end{equation}
so
\begin{equation}\label{eqn:4.63}
t_j\le t_j+u_j^2\le\log C.
\end{equation}
On the other hand, as $q(v_j)\ge1$, the third component shows that
\begin{equation}\label{eqn:4.64}
e^{-t_j}\le e^{-t_j}q(v_j)=Z_3(t_j,u_j,v_j)\le C,
\end{equation}
and thus $t_j\ge-\log C$. Obviously, $(t_j)$ is bounded. Hence
$((t_j,u_j,v_j))$ is bounded in $\mathbb{R}^3$.

Now, let $\Omega\subset\mathbb{R}^4$ be compact. Since $\Omega$
is bounded, the preceding argument shows that $Z^{-1}(\Omega)$
is bounded. Moreover, $Z^{-1}(\Omega)$ is closed because $Z$ is
continuous and $\Omega$ is closed. The Heine--Borel theorem
therefore implies that $Z^{-1}(\Omega)$ is compact in $\mathbb{R}^3$.
Thus $Z=F_3$ is proper.

It remains to verify injectivity. Note from the fact $Z_1,Z_3>0$ that
\begin{equation}\label{eqn:4.65}
\tfrac{Z_2}{Z_1}=f(u),\ \
\tfrac{Z_4}{Z_3}=\tfrac{v}{q(v)}.
\end{equation}
Both functions on the right-hand side are strictly increasing. Indeed,
\begin{equation}\label{eqn:4.66}
f'(u)=e^{-u^2}>0,\ \
(\tfrac{v}{q(v)})'
=\tfrac{q(v)-vq'(v)}{q(v)^2}
=\tfrac{e^{-v^2}}{q(v)^2}>0.
\end{equation}
Let $(t,u,v)$ and $(\tilde t,\tilde u,\tilde v)$ denote two distinct points,
and write $Z_k=Z_k(t,u,v)$ and $\tilde Z_k=Z_k(\tilde t,\tilde u,\tilde v)$.
If $u\ne\tilde u$, then the strict monotonicity of $f$ gives
\begin{equation}\label{eqn:4.67}
\tfrac{Z_2}{Z_1}=f(u)\ne f(\tilde u)
=\tfrac{\tilde Z_2}{\tilde Z_1},
\end{equation}
so the pairs $(Z_1,Z_2)$ and $(\tilde Z_1,\tilde Z_2)$ are different.
If $v\ne\tilde v$, then similarly the pairs $(Z_3,Z_4)$ and $(\tilde Z_3,
\tilde Z_4)$ are different. It remains to consider the case $u=\tilde u$
and $v=\tilde v$. Since the two parameter points are distinct, we have
$t\ne\tilde t$, and hence
\begin{equation}\label{eqn:4.68}
Z_1=e^{t+u^2}\ne e^{\tilde t+\tilde u^2}=\tilde Z_1.
\end{equation}
Consequently, any two distinct parameter points have distinct images,
and $Z$ is injective.

According to Step 1, we see from \eqref{eqn:4.45} that $[Z_t,Z_u,Z_v,Z]
\ne0$ and hence $Z$ is an immersion. We have now proved that $Z$ is both
injective and proper. By the standard fact that an injective proper immersion
into a Hausdorff manifold is an embedding, $Z=F_3$ is a proper embedding.

\smallskip
\emph{{\bf Step 4}: Show that $F_3$ is Euclidean complete.}

\smallskip
Steps 1 and 3 show that $F_3$ is a proper locally strongly convex
hyperbolic centroaffine immersion, so Lemma \ref{lem:proper-complete}
yields Euclidean completeness.\qed

\section{An $n$-dimensional complete centroaffine extremal hypersurface}\label{sect:5}

In this section, we extend $F_2$ and $F_3$ to higher dimensions by iterating the
Calabi composition with a point. After establishing its inheritance properties,
we prove Theorems \ref{thm:1.3} and \ref{thm:1.4} and compare the resulting
families with Wang's class to obtain Corollary \ref{cor:1.1}.

\subsection{Calabi composition of centroaffine hypersurface with a point}\label{sect:5.1}~

\vskip1mm
The metric formula and flatness criterion for the Calabi composition with a
point are recalled from \cite[Proposition~3.2]{CHM17}. We include the Tchebychev
function calculation and the global properties required for the iterations below.

\begin{proposition}\label{pro:5.1}
Let $\phi_1:M_1^{n_1}\to\mathbb{R}^{n_1+1}$ $(n_1\ge1)$ be a locally strongly
convex hyperbolic centroaffine hypersurface with the centroaffine metric $h_1$
and Tchebychev function $\psi_1$. For $\lambda>0$, the Calabi composition with
a point
\begin{equation}\label{eqn:5.1}
\widetilde{\mathcal C}_\lambda(t,p)
=\bigl(e^t\phi_1(p),e^{-\lambda t}\bigr):
\mathbb{R}\times M_1\to\mathbb R^{n_1+2}
\end{equation}
is a locally strongly convex hyperbolic centroaffine hypersurface with
centroaffine metric
\begin{equation}\label{eqn:5.2}
h_{\widetilde{\mathcal C}_\lambda}
=\lambda\,dt^2\oplus\tfrac{\lambda}{\lambda+1}h_1.
\end{equation}
Moreover, the following properties hold:
\begin{enumerate}
\item[$(1)$]
$h_{\widetilde{\mathcal C}_\lambda}$ is flat if and only if $h_1$ is flat.

\item[$(2)$]
$h_{\widetilde{\mathcal C}_\lambda}$ is complete if and only if $h_1$ is
complete.

\item[$(3)$]
The Tchebychev function is
\begin{equation}\label{eqn:5.3}
\psi_{\widetilde{\mathcal C}_\lambda}
=\tfrac{\lambda^{n_1+1}}{(\lambda+1)^{n_1+2}}
e^{-2(n_1+1-\lambda)t}\psi_1.
\end{equation}

\item[$(4)$]
Its Laplacian satisfies
\begin{equation}\label{eqn:5.4}
\Delta_{h_{\widetilde{\mathcal C}_\lambda}}
\log\psi_{\widetilde{\mathcal C}_\lambda}
=\tfrac{\lambda+1}{\lambda}\Delta_{h_1}\log\psi_1.
\end{equation}

\item[$(5)$]
$\widetilde{\mathcal C}_\lambda$ is centroaffine extremal if and only if
$\phi_1$ is centroaffine extremal.

\item[$(6)$]
$\widetilde{\mathcal C}_\lambda$ is an embedding if and only if $\phi_1$
is an embedding.

\item[$(7)$]
$\widetilde{\mathcal C}_\lambda$ is proper if and only if $\phi_1$ is
proper.

\item[$(8)$]
The Euclidean metric $g_{\widetilde{\mathcal C}_\lambda}$ induced by
$\widetilde{\mathcal C}_\lambda$ is complete if and only if the Euclidean
metric $g_1$ induced by $\phi_1$ is complete.
\end{enumerate}
\end{proposition}

\begin{proof}
For convenience, we write $Z=\widetilde{\mathcal C}_\lambda$, and choose
local coordinates $(p^1,\ldots,p^{n_1})$ on $M_1$. Putting $\phi_{1,i}=
\partial\phi_1/\partial p^i$, by differentiating \eqref{eqn:5.1} we find
\begin{equation}\label{eqn:5.5}
Z_t=(e^t\phi_1,-\lambda e^{-\lambda t}),\ \
Z_i=(e^t\phi_{1,i},0),
\end{equation}
and hence
\begin{equation}\label{eqn:5.6}
\begin{split}
[Z_t,Z_1,\ldots,Z_{n_1},Z]
=(-1)^{n_1}(\lambda+1)e^{(n_1+1-\lambda)t}
[\phi_{1,1},\ldots,\phi_{1,n_1},\phi_1]\ne0.
\end{split}
\end{equation}
Thus $\widetilde{\mathcal C}_\lambda:\mathbb R\times M_1\to
\mathbb{R}^{n_1+2}$ is a centroaffine hypersurface.

Moreover, applying the identities
\begin{equation}\label{eqn:5.7}
(e^t\phi_1,0)=\tfrac1{\lambda+1}Z_t+\tfrac\lambda{\lambda+1}Z,\ \
(0,e^{-\lambda t})=-\tfrac1{\lambda+1}Z_t+\tfrac1{\lambda+1}Z,
\end{equation}
and the Gauss formula for $\phi_1$, we deduce that
\begin{equation}\label{eqn:5.8}
\begin{aligned}
D_{\partial_t}Z_t&=(1-\lambda)Z_t+\lambda Z,\ \
D_{\partial_t}Z_i=D_{\partial_i}Z_t=Z_i,\\
D_{\partial_i}Z_j
&=Z_*(\nabla^1_{\partial_i}\partial_j)
+\tfrac{(h_1)_{ij}}{\lambda+1}Z_t
+\tfrac{\lambda(h_1)_{ij}}{\lambda+1}Z,
\end{aligned}
\end{equation}
where $\nabla^1$ denotes the centroaffine connection induced by $\phi_1$
and $(h_1)_{ij}:=h_1(\partial_i,\partial_j)$ with $\partial_i:=\partial/
\partial p^i$. Comparing these identities with \eqref{eqn:2.1}, we obtain
the expression in \eqref{eqn:5.2}. Since $\lambda>0$ and $h_1$ is positive
definite, $h_Z$ is positive definite. Hence $\widetilde{\mathcal C}_\lambda$
is locally strongly convex and hyperbolic.

It is known from \eqref{eqn:5.2} that the Levi--Civita connection of
$h_Z$ is the product connection. Its curvature is therefore the direct
sum of the curvature of the line and that of the constant rescaling
$\lambda h_1/(\lambda+1)$. Consequently, $h_Z$ is flat if and only if
$h_1$ is flat, proving the assertion (1). Since
\begin{equation}\label{eqn:5.9}
(\mathbb R\times M_1,h_Z)
=(\mathbb R,\lambda\,dt^2)
\times(M_1,\tfrac{\lambda}{\lambda+1}h_1)
\end{equation}
is a Riemannian product, it is complete if and only if $h_1$ is complete.
This proves the assertion (2). Using \eqref{eqn:5.2}, we further have
\begin{equation}\label{eqn:5.10}
\det(h_Z)=\tfrac{\lambda^{n_1+1}}{(\lambda+1)^{n_1}}\det(h_1).
\end{equation}
Combining this with \eqref{eqn:5.6} and \eqref{eqn:2.13}, we obtain
\begin{equation}\label{eqn:5.11}
\psi_Z
=\tfrac{\lambda^{n_1+1}}{(\lambda+1)^{n_1+2}}
e^{-2(n_1+1-\lambda)t}\psi_1,
\end{equation}
which proves (3). Furthermore, it can be checked from \eqref{eqn:5.2}
that
\begin{equation}\label{eqn:5.12}
\Delta_{h_Z}
=\tfrac1\lambda\partial_t^2
+\tfrac{\lambda+1}{\lambda}\Delta_{h_1}.
\end{equation}
The $t$-dependent term in $\log\psi_Z$ is linear. Applying the preceding
operator to \eqref{eqn:5.3}, we obtain \eqref{eqn:5.4}, proving (4). Since
$(\lambda+1)/\lambda>0$, equations \eqref{eqn:2.15} and \eqref{eqn:5.4} show
that $Z$ is centroaffine extremal if and only if $\phi_1$ is centroaffine
extremal. This proves (5).

Next, we shall establish the global assertions. Write
\begin{equation}\label{eqn:5.13}
Z(t,p)=(X(t,p),y(t)),\ \
X(t,p)=e^t\phi_1(p),\ \
y(t)=e^{-\lambda t}.
\end{equation}
Suppose first that $\phi_1$ is an embedding. If
$Z(t,p)=Z(\widetilde t,\widetilde p)$, then equality of the last coordinates
gives $t=\widetilde t$, while equality of the first $n_1+1$ coordinates gives
$\phi_1(p)=\phi_1(\widetilde p)$. Thus $p=\widetilde p$. Moreover,
the parameters can be recovered continuously from
\begin{equation}\label{eqn:5.14}
t=-\tfrac1\lambda\log y,\ \ p=\phi_1^{-1}(e^{-t}X).
\end{equation}
More precisely, the inverse of $Z$ on its image is
\begin{equation}\label{eqn:5.15}
Z^{-1}(X,y)
=\big(-\tfrac1\lambda\log y,
\phi_1^{-1}(e^{\frac1\lambda\log y}X)\big).
\end{equation}
This inverse is continuous because $y>0$ and $\phi_1^{-1}:\phi_1(M_1)\to
M_1$ is continuous. Since $Z$ is already an immersion, it is an embedding.
Conversely, suppose that $Z$ is an embedding, and consider the inclusion
\begin{equation}\label{eqn:5.16}
\sigma:M_1\to\mathbb{R}\times M_1,\ \
\sigma(p)=(0,p).
\end{equation}
The restriction of $Z$ to this slice is
\begin{equation}\label{eqn:5.17}
(Z\circ\sigma)(p)=(\phi_1(p),1).
\end{equation}
Since $Z$ is an embedding, their composition $Z\circ\sigma:M_1\to\mathbb
R^{n_1+2}$ is an embedding. Identifying $\mathbb R^{n_1+1}\times\{1\}$
with $\mathbb{R}^{n_1+1}$, we see that $\phi_1$ is an embedding. This
proves (6).

Suppose now that $\phi_1$ is proper, and let $(Z(t_j,p_j))$ be a bounded
sequence in $\mathbb R^{n_1+2}$. Choose $C>0$ such that
\begin{equation}\label{eqn:5.18}
|Z(t_j,p_j)|\le C
\end{equation}
for all $j$. Since $Z(t_j,p_j)=(X(t_j,p_j),y(t_j))$, both component
sequences are bounded. In particular,
\begin{equation}\label{eqn:5.19}
e^{-\lambda t_j}=y(t_j)\le C,
\end{equation}
and hence
\begin{equation}\label{eqn:5.20}
t_j\ge-\tfrac1\lambda\log C.
\end{equation}
Consequently, $(t_j)$ is bounded from below. We claim that it is also
bounded from above. Otherwise, after passing to a subsequence, $t_j\to
\infty$. Since $|X(t_j,p_j)|\le C$, we have
\begin{equation}\label{eqn:5.21}
|\phi_1(p_j)|=|X(t_j,p_j)|e^{-t_j}\le Ce^{-t_j}\to 0.
\end{equation}
The sequence $(\phi_1(p_j))$ is therefore contained in a closed ball
$\bar B\subset\mathbb R^{n_1+1}$. Furthermore, this ball is compact,
and the properness of $\phi_1$ implies that $\phi_1^{-1}(\bar B)$ is
compact in $M_1$. Consequently, after passing to a further subsequence,
$p_j\to p_\infty$ for some $p_\infty\in M_1$. By the continuity of
$\phi_1$ and \eqref{eqn:5.21}, we obtain $\phi_1(p_\infty)=0$. This
is impossible, because
\begin{equation}\label{eqn:5.22}
[\phi_{1,1},\ldots,\phi_{1,n_1},\phi_1](p_\infty)=0
\end{equation}
is a contradiction to centroaffine transversality. Hence $(t_j)$
is bounded. It follows from
\begin{equation}\label{eqn:5.23}
\phi_1(p_j)=e^{-t_j}X(t_j,p_j)
\end{equation}
that $(\phi_1(p_j))$ is bounded. Thus it is contained in a compact closed
ball $\bar B_1\subset\mathbb R^{n_1+1}$, and all $p_j$ lie in the compact
set $\phi_1^{-1}(\bar B_1)$. Therefore $(t_j,p_j)$ is contained in a compact
subset of $\mathbb R\times M_1$. Now let $\Omega\subset\mathbb{R}^{n_1+2}$
be compact. Applying the preceding estimates uniformly to the points of
$Z^{-1}(\Omega)$, we obtain a compact interval $I\subset\mathbb R$ and a
compact subset $\Lambda_1\subset M_1$ such that
\begin{equation}\label{eqn:5.24}
Z^{-1}(\Omega)\subset I\times\Lambda_1.
\end{equation}
Moreover, $Z^{-1}(\Omega)$ is closed because $Z$ is continuous and $\Omega$
is closed. It is a closed subset of the compact set $I\times\Lambda_1$, and
hence is compact. Thus $Z$ is proper.

Conversely, suppose that $Z$ is proper, and let $\Lambda\subset
\mathbb{R}^{n_1+1}$ be compact. Then $\Lambda\times\{1\}$ is compact in
$\mathbb R^{n_1+2}$, so $Z^{-1}(\Lambda\times\{1\})$ is compact. If $Z(t,p)
\in \Lambda\times\{1\}$, then the last coordinate gives $e^{-\lambda t}=1$.
Since $\lambda>0$, we obtain $t=0$, and the first $n_1+1$ coordinates then
give $\phi_1(p)\in \Lambda$. Conversely, if $p\in\phi_1^{-1}(\Lambda)$,
then
\begin{equation}\label{eqn:5.25}
Z(0,p)=(\phi_1(p),1)\in \Lambda\times\{1\}.
\end{equation}
It follows that
\begin{equation}\label{eqn:5.26}
Z^{-1}(\Lambda\times\{1\})=\{0\}\times\phi_1^{-1}(\Lambda).
\end{equation}
The natural projection $\{0\}\times M_1\to M_1$ is a homeomorphism,
and hence the compactness of the left-hand side implies that $\phi_1^{-1}
(\Lambda)$ is compact. Since this holds for every compact $\Lambda\subset
\mathbb R^{n_1+1}$, the map $\phi_1$ is proper. This proves (7).

Finally, by Lemma \ref{lem:proper-complete}, both $\phi_1$ and $Z$ are
Euclidean complete if and only if they are proper. Combining this with (7),
we conclude that the Euclidean metric induced by $Z$ is complete if and
only if $g_1$ is complete. This proves (8) for every $\lambda>0$.
\end{proof}

\subsection{Completion of the proof of Theorem \ref{thm:1.3}}\label{sect:5.2}~

\vskip1mm
We argue by induction on $n$. It is seen from Theorem \ref{thm:1.1}
that the base immersion $G_2=F_2$ is a smooth proper embedding and a
locally strongly convex hyperbolic centroaffine extremal surface whose
Euclidean induced metric is complete and whose centroaffine metric is
incomplete. Suppose that $G_{n-1}$ has these properties, and set $G_n
=\widetilde{\mathcal C}_1(G_{n-1})$ as in \eqref{eqn:1.5}. Together with
$\lambda=1$ and $\phi_1=G_{n-1}$, Proposition \ref{pro:5.1} shows that
$G_n$ is again a smooth proper embedding and a locally strongly convex
hyperbolic centroaffine extremal hypersurface. Moreover, the assertions
$(2)$ and $(8)$ preserve the incompleteness of the centroaffine metric
and the completeness of the Euclidean induced metric, respectively.

We now verify that the centroaffine metric $h_n$ induced by $G_n$ has
nonnegative Ricci curvature. By means of \eqref{eqn:5.2}, it can be checked that
\begin{equation}\label{eqn:5.27}
h_n=dr_n^2\oplus\tfrac12h_{n-1}.
\end{equation}
From this, $h_n$ is a Riemannian product of a flat line and a constant
rescaling of $h_{n-1}$. Since $\operatorname{Ric}_{h_2}=K_{h_2}h_2$ and
$K_{h_2}>0$ by Theorem \ref{thm:1.1}, the induction shows that the Ricci
curvature of $h_n$ is nonnegative. This completes the proof of Theorem
\ref{thm:1.3}.\qed

\subsection{Completion of the proof of Theorem \ref{thm:1.4}}\label{sect:5.3}~

\vskip1mm
We shall argue by induction on $n$. According to Theorem \ref{thm:1.2},
the base hypersurface $F_3$ has all the required properties. Suppose
that $F_{n-1}:\mathbb{R}^{n-1}\to\mathbb{R}^n$ is a smooth proper embedding
and a locally strongly convex hyperbolic centroaffine extremal hypersurface
whose centroaffine metric is flat and complete and whose Euclidean induced
metric is complete. Set $F_n=\widetilde{\mathcal C}_1(F_{n-1})$ in \eqref{eqn:1.6}.
Together with $\lambda=1$ and $\phi_1=F_{n-1}$, Proposition \ref{pro:5.1}
shows that $F_n$ is a locally strongly convex hyperbolic centroaffine
hypersurface, and the assertion $(5)$ preserves the centroaffine extremality.
Moreover, the assertions $(1)$--$(2)$ preserve the flatness and completeness
of the centroaffine metric, and the assertions $(6)$--$(7)$ preserve the
embedding and properness. Finally, the Euclidean completeness follows from
the assertion $(8)$. Consequently, $F_n$ has all the stated properties.
Starting from $F_3$ and repeating this argument proves Theorem \ref{thm:1.4}
for every $n\geq4$.\qed

\subsection{Completion of the proof of Corollary \ref{cor:1.1}}\label{sect:5.4}~

\vskip1mm
It remains to rule out the centroaffine equivalence to Wang's class.
Cheng--Hu--Xing \cite[Example 3.1 and Claim 3.1]{CHX22} proved that the
hypersurfaces in \eqref{eqn:1.1} have flat centroaffine metric and parallel
difference tensor. Taking the $h$-trace of $\widehat\nabla K=0$ yields $\mathcal
T=\widehat\nabla T=0$. Flatness and the vanishing of $\mathcal T$ are invariant
under centroaffine equivalence. The family $G_n$, including $G_2=F_2$, is
nonflat: the positive curvature of its two-dimensional factor is preserved
by the constant rescalings in \eqref{eqn:5.27}. This excludes $G_n$ from
Wang's class. The flat family $F_n$ with $n\ge3$ is distinguished instead
by its Tchebychev operator. From the combination of \eqref{eqn:4.50},
\eqref{eqn:5.2} and \eqref{eqn:5.3}, the metric on the original $u,v$
directions is $2^{3-n}(du^2+dv^2)$, while $\log\psi_{F_n}$ retains the
terms $-2u^2+2v^2$ and is affine in the remaining coordinates. According
to \eqref{eqn:2.14}, it holds that
\begin{equation}\label{eqn:5.28}
\mathcal T_{F_n}(\partial_u)=\tfrac{2^{n-2}}{n}\partial_u,\ \
\mathcal T_{F_n}(\partial_v)=-\tfrac{2^{n-2}}{n}\partial_v.
\end{equation}
Thus $\mathcal T_{F_n}\ne0$, so $F_n$ cannot be centroaffinely equivalent
to a member of Wang's class either. Together with
Theorems \ref{thm:1.1}--\ref{thm:1.4}, these observations prove
Corollary \ref{cor:1.1}.\qed

\vskip2mm
\noindent{\bf Acknowledgments.} The authors would like to thank Profs. Zejun Hu, Haizhong Li and
Luc Vrancken for their kind interest in this work and for their encouraging comments on an earlier draft.  
The project was supported by the National Natural Science Foundation of China (Grant No. 12501063), 
the Postdoctoral Fellowship Program of CPSF (Grant No. GZC20252036), China Postdoctoral Science Foundation (Grant No.
2025M783098), and Natural Science Foundation of Henan Province (Grant No. 262300421852). 

\vskip2mm
\noindent{\bf AI assistance statement.} The authors used AI models to assist with proving Theorem 1.1;
the main ideas and all final checks remain the sole responsibility of the human authors.

\vskip2mm
\noindent{\bf Competing interests.}
The authors declare that there are no conflicts of interest.




\begin{thebibliography}{99}

\bibitem{Alm66}
Almgren, F.J., Jr.:
Some interior regularity theorems for minimal surfaces and an extension of
Bernstein's theorem.
Ann. of Math. (2) \textbf{84}, 277--292 (1966)

\bibitem{AC00}
Ambrosio, L., Cabr\'{e}, X.:
Entire solutions of semilinear elliptic equations in $\mathbb R^3$ and a
conjecture of De Giorgi.
J. Amer. Math. Soc. \textbf{13}, 725--739 (2000)

\bibitem{Ber15}
Bernstein, S.:
Sur un th\'eor\`eme de g\'eom\'etrie et son application aux \'equations aux
d\'eriv\'ees partielles du type elliptique.
Comm. Soc. Math. Kharkov (2) \textbf{15}, 38--45 (1915)

\bibitem{BDGG69}
Bombieri, E., De Giorgi, E., Giusti, E.:
Minimal cones and the Bernstein problem.
Invent. Math. \textbf{7}, 243--268 (1969)

\bibitem{Cal82}
Calabi, E.:
Hypersurfaces with maximal affinely invariant area.
Amer. J. Math. \textbf{104}, 91--126 (1982)

\bibitem{Cal88}
Calabi, E.:
Convex affine maximal surfaces.
Results Math. \textbf{13}, 199--223 (1988)

\bibitem{CHM17}
Cheng, X., Hu, Z., Moruz, M.:
Classification of the locally strongly convex centroaffine hypersurfaces with parallel cubic form.
Results Math. \textbf{72}, 419--469 (2017)

\bibitem{CHX22}
Cheng, X., Hu, Z., Xing, C.:
On centroaffine Tchebychev hypersurfaces with constant sectional curvature.
Results Math. \textbf{77}, Paper No.~175, 29 pp. (2022)

\bibitem{Che79}
Chern, S.-S.:
Affine minimal hypersurfaces.
In: Obata, M. (ed.) Minimal Submanifolds and Geodesics, pp. 17--30.
North-Holland, Amsterdam (1979)

\bibitem{ChL24}
Chodosh, O., Li, C.:
Stable minimal hypersurfaces in $\mathbb R^4$.
Acta Math. \textbf{233}, 1--31 (2024)

\bibitem{DG65}
De Giorgi, E.:
Una estensione del teorema di Bernstein.
Ann. Scuola Norm. Sup. Pisa Cl. Sci. (3) \textbf{19}, 79--85 (1965)

\bibitem{DJX24}
Ding, Q., Jost, J., Xin, Y.:
Minimal graphs of arbitrary codimension in Euclidean space with bounded
$2$-dilation.
Math. Ann. \textbf{390}, 3355--3418 (2024)

\bibitem{HLW26} Hong, H., Li, H., Wang, G.M.:
The stable Bernstein theorem in  $\mathbb {R}^7$.  arXiv:2609.15720v1 (2026)

\bibitem{JWZ23}
Jiang, G., Wang, Z., Zhu, J.:
Stability of Edelen--Wang's Bernstein type theorem for the minimal surface
equation. J. Funct. Anal. \textbf{284}, Article 109821 (2023)


\bibitem{JX02}
Jost, J., Xin, Y.-L.:
A Bernstein theorem for special Lagrangian graphs.
Calc. Var. Partial Differential Equations \textbf{15}, 299--312 (2002)

\bibitem{LXZ25}
Lei, M., Xu, R., Zhao, P.:
Classification of the conformally flat centroaffine hypersurfaces with
vanishing centroaffine shape operator.
Chinese Ann. Math. Ser. B \textbf{46}, 163--180 (2025)

\bibitem{LJ01}
Li, A.-M., Jia, F.:
The Calabi conjecture on affine maximal surfaces.
Results Math. \textbf{40}, 265--272 (2001)


\bibitem{LJ09}
Li, A.-M., Jia, F.:
A Bernstein property of some fourth order partial differential equations.
Results Math. \textbf{56}, 109--139 (2009)

\bibitem{LLS04}
Li, A.-M., Li, H., Simon, U.:
Centroaffine Bernstein problems.
Differential Geom. Appl. \textbf{20}, 331--356 (2004)

\bibitem{LLY20}
Li, D., Li, Z., Yuan, Y.:
A Bernstein problem for special Lagrangian equations in exterior domains.
Adv. Math. \textbf{361}, Article 106927 (2020)

\bibitem{LSZH15}
Li, A.-M., Simon, U., Zhao, G., Hu, Z.:
Global Affine Differential Geometry of Hypersurfaces, 2nd ed.
De Gruyter Expositions in Mathematics, vol. 11. Walter de Gruyter,
Berlin/Boston (2015)

\bibitem{LW91}
Li, A.-M., Wang, C.P.:
Canonical centroaffine hypersurfaces in $\mathbb{R}^{n+1}$.
Results Math. \textbf{20}, 660--681 (1991)

\bibitem{LSW96}
Liu, H., Simon, U., Wang, C.P.:
Conformal structures in affine geometry: complete Tchebychev hypersurfaces.
Abh. Math. Semin. Univ. Hambg. \textbf{66}, 249--262 (1996)

\bibitem{LW95}
Liu, H., Wang, C.P.:
The centroaffine Tchebychev operator.
Results Math. \textbf{27}, 77--92 (1995)

\bibitem{MY24}
Mooney, C., Yang, Y.:
The anisotropic Bernstein problem.
Invent. Math. \textbf{235}, 211--232 (2024)

\bibitem{NS94}
Nomizu, K., Sasaki, T.:
Affine Differential Geometry.
Cambridge University Press, Cambridge (1994)

\bibitem{Rud76}
Rudin, W.:
Principles of Mathematical Analysis, 3rd ed.
McGraw--Hill, New York (1976)

\bibitem{Sav09}
Savin, O.:
Regularity of flat level sets in phase transitions.
Ann. of Math. (2) \textbf{169}, 41--78 (2009)

\bibitem{Sim68}
Simons, J.:
Minimal varieties in Riemannian manifolds.
Ann. of Math. (2) \textbf{88}, 62--105 (1968)

\bibitem{SXX26}
Sun, Y., Xing, C., Xu, R.:
Calabi affine maximal surfaces and centroaffine Bernstein problems.
arXiv:2601.10125v2 (2026)

\bibitem{TW00}
Trudinger, N.S., Wang, X.-J.:
The Bernstein problem for affine maximal hypersurfaces.
Invent. Math. \textbf{140}, 399--422 (2000)

\bibitem{Wan94}
Wang, C.P.:
Centroaffine minimal hypersurfaces in $\mathbb{R}^{n+1}$.
Geom. Dedicata \textbf{51}, 63--74 (1994)

\end{thebibliography}
\end{document}